\documentclass{siamart171218}

\RequirePackage{amsmath}
\RequirePackage{amsfonts}
\RequirePackage{amssymb}
\RequirePackage{hyperref}
\RequirePackage{multirow}
\RequirePackage{algorithm}
\RequirePackage{algpseudocode}
\RequirePackage{graphicx}
\RequirePackage{xcolor}
\allowdisplaybreaks

\usepackage{dpr_fec,ifthen}
\usepackage{csvsimple} 
\usepackage{booktabs}

\newtheorem{assumption}{Assumption}[section]

\title{A Proximal-Gradient Method for Constrained Optimization\footnotemark[1]}
\author{
   Yutong Dai\!\! 
   \and Xiaoyi Qu\!\! 
   \and Daniel P.~Robinson\footnotemark[2]}
\date  {\today}

\begin{document}

\maketitle

\renewcommand{\thefootnote}{\fnsymbol{footnote}}
\footnotetext[1]{This work is supported by the U.S. Office of Naval Research under Award Number N00014-21-1-2532 and by the U.S.~National Science Foundation, Division of Mathematical Sciences, Computational Mathematics Program under Award Number DMS--2012243.}
\footnotetext[2]{All authors are from the Department of Industrial and Systems Engineering, Lehigh University, Bethlehem, PA, USA; E-mail: \email{yud319@lehigh.edu}, \email{xiq322@lehigh.edu}, \email{daniel.p.robinson@lehigh.edu}}
\renewcommand{\thefootnote}{\arabic{footnote}}

\begin{abstract}
  We present a new algorithm for solving optimization problems with objective functions that are the sum of a smooth function and a (potentially) nonsmooth regularization function, and nonlinear equality constraints.  The algorithm may be viewed as an extension of the well-known proximal-gradient method that is applicable when constraints are not present.  To account for nonlinear equality constraints, we combine a decomposition procedure for computing trial steps with an exact merit function for determining trial step acceptance.  Under common assumptions, we show that both the proximal parameter and merit function parameter eventually remain fixed, and then prove a worst-case complexity result for the maximum number of iterations before an iterate satisfying approximate first-order optimality conditions for a given tolerance is computed. Our preliminary numerical results indicate that our approach has great promise, especially in terms of returning approximate solutions that are structured (e.g., sparse solutions when a one-norm regularizer is used).







\end{abstract}

\begin{keywords}
  nonlinear optimization, nonconvex optimization, worst-case iteration complexity, regularization methods, sequential quadratic programming, sequential quadratic optimization
\end{keywords}

\begin{AMS}
  49M37, 65K05, 65K10, 65Y20, 68Q25, 90C30, 90C60
\end{AMS}

\numberwithin{equation}{section}
\numberwithin{theorem}{section}

\newcommand{\tauktrial}{\tau_{k,\text{trial}}}
\newcommand{\tautrialmin}{\tau_{\text{min,trial}}}
\newcommand{\xtrial}{x^{\text{trial}}}
\newcommand{\finf}{f_{\text{inf}}}
\newcommand{\phibar}{\bar\phi}
\newcommand{\alphamin}{\alpha_{\min}}
\newcommand{\taumin}{\tau_{\min}}
\newcommand{\sigmamin}{\sigma_{\min}}
\newcommand{\sigmamax}{\sigma_{\max}}
\newcommand{\lipg}{L_{g}}
\newcommand{\lipJ}{L_{J}}
\newcommand{\kappag}{\kappa_{\nabla f}}
\newcommand{\kappagr}{\kappa_{\partial r}}
\newcommand{\kappac}{\kappa_{c}}
\newcommand{\kappav}{\kappa_{v}}
\newcommand{\kappaJ}{\kappa_{\nabla c}}
\newcommand{\kappaPhi}{\kappa_{\Phi}}
\newcommand{\kappabarPhi}{\bar\kappa_{\Phi}}
\newcommand{\sigmabar}{\bar\sigma}
\newcommand{\epsilonbar}{\bar\epsilon}
\newcommand{\taubar}{\bar\tau}
\newcommand{\taubarmin}{\taubar_{\min}}
\newcommand{\alphabar}{\bar\alpha}
\newcommand{\alphabarmin}{\bar\alpha_{\min}}
\newcommand{\xstar}{x_*}
\newcommand{\ystar}{y_*}
\newcommand{\chibar}{\bar\chi}
\newcommand{\fus}{f_{\text{Algorithm~\ref{alg:main}}}}
\newcommand{\fBazinga}{f_{\text{Bazinga}}}

\section{Introduction}\label{sec.introduction}

In this paper we consider the problem
\begin{equation}\label{prob:main}
\min_{x\in\R{n}} f(x) + r(x) \ \ \text{subject to (s.t.)}  \ \ c(x) = 0,
\end{equation}
where $f:\R{n}\to \R{}$ is continuously differentiable, $r:\R{n} \to \R{}$ is convex (possibly nondifferentiable) and nonnegative valued, and $c:\R{n}\to\R{m}$ is continuously differentiable with $m \leq n$. The optimization problem~\eqref{prob:main} has  applications in model predictive control~\cite{annergren2012admm}, image processing~\cite{yuan2017ell}, nonsmooth optimization on a Steifel manifold~\cite{xiao2021penalty}, and low rank matrix completion~\cite{cai2010singular}. In addition, optimization problems such as sparse approximation, empirical risk minimization, and neural network modeling with mixed activations  can be reformulated as~\eqref{prob:main}; see~\cite{zhu2023first} for additional details.

When the regularizer $r$ is not present, the algorithms most commonly employed to solve problem~\eqref{prob:main} are penalty methods~\cite{ConGT91c,ConP77,CurJR13,FiaM90,MayM79,Rob13,zavala2014scalable} and sequential quadratic optimization (SQO) methods~\cite{Ani02,GouR10a,GouR10b,GurO89,han1977globally,han1979exact,MorNW12,powell2006fast}.  Penalty methods is to solve problem~\eqref{prob:main} by minimizing a sequence of unconstrained optimization subproblems defined in terms of $f$, a measure of constraint violation, and various parameters (e.g., Lagrange multiplier estimates and penalty parameters).  After each minimization subproblem in the sequence is solved, the parameters are updated in a manner that allows for convergence guarantees.  Since computing each subproblem minimizer may be expensive, and the number of subproblems solved may be nontrivial, penalty methods often require a significant amount of computation (e.g., numbers of iterations, function/derivative evaluations, and linear systems solved), which may be  prohibitive.
%
On the other hand, during each iteration of a line search SQO method, the main expense is the computation of a search direction, which is achieved by solving a single linear system of equations.  Equivalently, the search direction is the minimizer of a certain quadratic approximation of $f$ subject to a linearization of the constraints. SQO methods are generally viewed as the state-of-the-art because of their remarkable practical performance.  The superior performance of line search SQO methods over penalty methods can be attributed to two main sources.  First, line search SQO methods solve a sequence of linear systems rather than a sequence of general optimization subproblems.  Second, the search directions for SQO methods are designed to find a solution of problem~\eqref{prob:main} (again, when $r$ is not present), whereas penalty methods \emph{indirectly} aim to find a solution of problem~\eqref{prob:main} (again, when $r$ is not present) by adjusting its parameters after minimizing each subproblem.  

When the constraint $c(x) = 0$ is not present in problem~\eqref{prob:main}, 
the algorithms most commonly employed are variants of the proximal-gradient (PG) method~\cite{beck2017first,BecT09,CheCR17,CheCR18,KariNutiSchm16,lee2019inexact}. Each iterate of a basic PG method is the minimizer of a subproblem (i.e., the PG subproblem) formed by replacing $f$ in~\eqref{alg:main} by the sum of its first-order Taylor expansion (expanded at the current point) and a simple quadratic-regularization term.  For some commonly used regularizers, the PG subproblem has a closed-form solution, which is an attractive feature of such methods.  Moreover, since the regularizer $r$ is explicitly used in the definition of the PG subproblem (i.e., it is not approximated), the solutions generated by a PG method inherit the structure induced by the regularizer (e.g., if $r(x) = \|x\|_1$, then a PG method can produce sparse solution estimates).  This \emph{structure preserving} property is an important feature of PG methods when used to solve problem~\eqref{prob:main} (again, when the constraint $c(x) = 0$ is not present).

The work in this paper is motivated by both 
SQO methods for solving~\eqref{prob:main} when $r$ is not present, and the structure preserving property of PG methods for solving~\eqref{prob:main} when the constraint $c(x) = 0$ is not present.  In particular, we design and analyze a method for solving problem~\eqref{prob:main} based on subproblems that linearize the constraints (like SQO methods) and explicitly use the regularizer (like PG methods).



\subsection{Literature review}

We are aware of four papers, namely~\cite{de2023constrained, fukushima1990successive, liu2016smoothing, zhu2023first}, that present algorithms for minimizing  regularized optimization functions subject to nonlinear constraints.
The algorithms in~\cite{de2023constrained, zhu2023first} are penalty methods built upon the popular augmented Lagrangian function.  Therefore, both approaches have a penalty parameter and a vector of Lagrange multiplier estimates that balance the objective and constraint functions, and must be updated throughout the optimization procedure. We note that \cite{de2023constrained} can solve regularized optimization problems with both equality and inequality constraints, whereas the algorithm in~\cite{zhu2023first} can only handle  special classes of regularized optimization problems with constraints. The algorithms presented  in~\cite{fukushima1990successive, liu2016smoothing} are of the SQO variety. In~\cite{fukushima1990successive}, subgradient information of the nonsmooth function is used to formulate a sequence of min-max subproblems.  Since the regularizer is approximated in each subproblem, the structure preserving property of the iterates is lost. In contrast, \cite{liu2016smoothing} relies on a smoothing technique that approximates the nonsmooth term in the objective function and, thereafter, sequentially solves a convex quadratic problem with linear constraints. Unfortunately, in general, the smoothing technique  ruins the structure of the composite optimization problem, and consequently the structure preserving property is lost. 



\subsection{Contributions}

Our contributions relate to the proposal and analysis of a new algorithm for solving problem~\eqref{prob:main}, as we summarize next.

\begin{itemize}
\item We propose a PG-based algorithm for solving problem~\eqref{prob:main} that uses subproblems with linearized constraints (like SQO methods) and explicit regularization (like PG methods).  The method that results from this combination avoids the previously discussed challenges and weaknesses of augmented Lagrangian approaches, and provides solution estimates that are structure preserving.
During each iteration, we compute a trial step as the sum of two orthogonal directions called the normal and tangential steps.  First, the normal step is computed from a trust region subproblem designed to reduce the constraint violation.  Second, the tangential step is computed from a linearly constrained convex optimization subproblem with objective function reminiscent of PG methods (i.e., $r$ appears explicitly and a proximal term is used). Overall, the tangential step aims to reduce the objective function while maintaining the predicted progress in reducing infeasibility achieved by the normal step. The quality of the trial step, defined as the sum of the normal and tangential steps, is then determined by an $\ell_2$ merit function that uses a merit parameter to weight the objective function relative to the two-norm of the constraint violation.  The merit parameter and PG parameter (i.e., the weight on the proximal term) are reduced as the iterations proceed, if necessary, to promote convergence of the iterates to a solution of problem~\eqref{prob:main}.
\item Under minimal assumptions, we prove that a measure of first-order optimality for a feasibility problem converges to zero. Under additional commonly used assumptions, we prove that the merit parameter and PG parameter both remain uniformly bounded away from zero.  These results allow us to then prove that our algorithm 
generates a sequence of iterates such that any limit point is a KKT point (see Theorem~\ref{thm:limit-kkt}).  Moreover, we provide a worst-case complexity result for the maximum number of iterations before a certain criticality measure will be less than a given tolerance (see Theorem~\ref{thm:complexity}).
%
%
\item We present numerical experiments that verify our theoretical convergence results, and illustrate that our algorithm is capable of returning solutions that preserve the structure related to  $r$.  Specifically, we confirm that our method returns sparse solution estimates when $r$ is chosen as the $\ell_1$-norm function, which is known to be a sparsity-inducing regularizer.
\end{itemize}

\subsection{Notation and assumptions}

We use $\R{}$ to denote the set of real numbers (i.e., scalars), $\R{}_{\geq0}$ (resp.,~$\R{}_{>0}$) to denote the set of nonnegative (resp.,~positive) real numbers,~$\R{n}$ to denote the set of $n$-dimensional real vectors, and $\R{m \times n}$ to denote the set of $m$-by-$n$-dimensional real matrices.  The set of natural numbers is denoted as $\N{} := \{0,1,2,\dots\}$. Given a matrix $M\in\R{m\times n}$, we let $\sigmamin(M)$ (resp.,~$\sigmamax(M)$) denote the smallest (resp.,~largest) singular value of $M$.  For $v\in\R{n}$, we let $\|v\|_2 := \sqrt{\sum_{i=1}^n v_i^2}$ denote its two norm.  For a nonempty compact set $\Rcal\subset\R{n}$, we let $\|\Rcal\|_2 := \max\{\|s\|_2: s\in\Rcal\}$ denote its largest element measured in the two-norm.

The following assumption is used throughout the paper.

\bassumption\label{ass:basic}
  Let $\Xcal \subset \R{n}$ be an open convex set that contains the iterates $\{x_k\} \subset \R{n}$ and trial steps $\{x_k+s_k\}\subset \R{n}$ generated by Algorithm~\ref{alg:main}.  The function $f : \R{n} \to \R{}$ is continuously differentiable and bounded below over $\Xcal$ and its gradient function $\nabla f : \R{n} \to \R{n}$ is Lipschitz continuous and bounded over~$\Xcal$.  Similarly, the function $c : \R{n} \to \R{m}$ is continuously differentiable and bounded over $\Xcal$ and its Jacobian $\nabla c(x)^T$ is Lipschitz continuous and bounded over $\Xcal$. Finally, the function $r : \R{n} \to \R{}_{\geq 0}$ is convex and its subdifferential $\partial r: \R{n} \to \R{n}$ is bounded over $\Xcal$.  
\eassumption

Under Assumption~\ref{ass:basic}, there exist constants  $(\finf,\kappag,\kappagr,\kappac,\kappaJ,\lipg,\lipJ) \in \R{} \times \R{}_{>0} \times \R{}_{>0} \times \R{}_{>0} \times \R{}_{> 0} \times \R{}_{>0} \times \R{}_{> 0}$ such that for all $x\in\Xcal$ one has
\begin{equation}\label{eq:consequence-of-basic-assumption}
\begin{aligned}
f(x) &\geq \finf, & \ \ 
\|\nabla f(x)\|_2 &\leq \kappag, & \ \
\|\partial r(x)\|_2 &\leq \kappagr, \\ 
\|c(x)\|_2 &\leq \kappac, & \ \ 
\|\nabla c(x)^T\|_2 &\leq \kappaJ,
\end{aligned}
\end{equation}
and for all $(x,\xbar)\in\Xcal\times \Xcal$ one has
\begin{equation}\label{eq.Lipschitz}
  \|\nabla f(x) - \nabla f(\xbar)\|_2 \leq \lipg \|x - \xbar\|_2
  \ \ \text{and}\ \
  \|\nabla c(x)^T - \nabla c(\xbar)^T\|_2 \leq \lipJ \|x - \xbar\|_2.
\end{equation}

For convenience, we define $g(x) :=  \nabla f(x)$ and $J(x) := \nabla c(x)^T$. We append a natural number as a subscript for a quantity to denote its value during an iteration of an algorithm; i.e., we let $f_k := f(x_k)$, $g_k := g(x_k)$, $c_k := c(x_k)$, and $J_k := J(x_k)$.

\subsection{Organization}

In Section~\ref{sec:algorithm}, we propose our algorithm for solving problem~\eqref{prob:main}, and its  convergence properties are presented in Section~\ref{sec:analysis}.  In Section~\ref{sec:numerical}, we discuss our  numerical tests.  Final conclusions are provided in Section~\ref{sec:conclusion}.

\section{Algorithm}\label{sec:algorithm}

The algorithm that we propose for solving problem~\eqref{prob:main} is formally stated as Algorithm~\ref{alg:main}.  Given the $k$th iterate $x_k$, the $k$th PG parameter $\alpha_k$, and constant $\kappav \in \R{}_{> 0}$, we compute a step $v_k$ that aims to reduce the constraint infeasibility at $x_k$ as an approximate solution to the following problem:
\begin{equation}\label{prob:normal.step}
    \min_{v\in\R{n}} \ m_k(v)
    \ \st \ \|v\|_2 \leq 
    \kappav\alpha_k\|J_k^T c_k\|_2, \ \ \text{with} \ \ m_k(v) := \tfrac{1}{2}\|c_k+J_kv\|_2^2.
\end{equation}
The PG parameter $\alpha_k$ is used to define the trust-region constraint so that $\{v_k\} \to 0$ if $\{\alpha_k\} \to 0$. We consider a vector $v_k$ to be an adequate approximate solution to subproblem~\eqref{prob:normal.step} if, for some 
$\kappav\in \R{}_{>0}$, it satisfies the following  conditions:
\begin{subequations}\label{conds:vk}
\begin{align}
v_k &\in \Range(J_k^T), \label{conds:vk-1}\\
\|v_k\|_2 &\leq \kappav\alpha_k\|J_k^T c_k\|_2,
\ \text{and} \label{conds:vk-2}\\
\|c_k + J_k v_k\|_2 &\leq \|c_k+J_k v_k^c\|_2 \label{conds:vk-3}
\end{align}
\end{subequations}
where $v_k^c$ is the Cauchy point given by
\begin{equation}\label{def:vkc}
v_k^c:= -\beta_k^c J_k^T c_k
\ \ \text{with} \ \ 
\beta_k^c := \arg\min_{\beta\in\R{}} \ m_k(-\beta J_k^T c_k) \ \ \st \ \ 0 \leq \beta \leq \kappav\alpha_k. 
\end{equation}
In other words, the Cauchy point $v_k^c$ minimizes $m_k(v)$ along the direction $-\nabla m_k(0) = -J_k^T c_k$ and within $\{v: \|v\|_2 \leq \kappav\alpha_k\|J_k^T c_k\|_2\}$.  It is known (see~\cite{ConGT00a}) that $v_k^c$ satisfies 
\begin{equation}\label{eq:cauchy-decrease}
m_k(0) - m_k(v_k^c)
\geq \thalf \|J_k^T c_k\|_2^2 \min\left\{ \tfrac{1}{1+\|J_k^T J_k\|_2}, \kappav\alpha_k\right\}. 
\end{equation}
We note that the conditions~\eqref{conds:vk} are well-posed since they are satisfied by $v_k = v_k^c$.


Next, we compute a direction $u_k$ that maintains the level of linearized infeasibility achieved by $v_k$ while also reducing a model of the objective function.  Specifically, we compute $u_k$ as the unique solution to the strongly convex subproblem
\begin{equation}\label{prob:prox.sqp.relaxed}
\begin{aligned}
    u_k &:= \arg\min_{u\in\R{n}} \ g_k^T (v_k+u) + \tfrac{1}{2\alpha_k}\|v_k + u\|_2^2 + r(x_k+v_k + u) \ \ \text{s.t.} \ \ J_k u = 0. \\ 
    &\phantom{:}= \arg\min_{u\in\R{n}} \ (g_k+\tfrac{1}{\alpha_k}v_k)^T u + \tfrac{1}{2\alpha_k}\|u\|_2^2 + r(x_k+v_k + u) \ \ \text{s.t.} \ \ J_k u = 0 \\
    &\phantom{:}= \arg\min_{u\in\R{n}} \ g_k^T u + \tfrac{1}{2\alpha_k}\|u\|_2^2 + r(x_k+v_k + u) \ \ \text{s.t.} \ \ J_k u = 0
\end{aligned}
\end{equation}
where we used the fact that every $u$ feasible for~\eqref{prob:prox.sqp.relaxed} satisfies $v_k^T u = 0$ since $v_k\in\Range(J_k^T)$ (see~\eqref{conds:vk-1}).
The trial step $s_k$ is then defined as
\begin{equation}\label{def:sk}
s_k := v_k + u_k.
\end{equation}

We adopt the $\ell_2$ merit function, which for parameter $\tau \in\R{}_{>0}$ is defined as 
$$
\Phi_\tau(x) := \tau\big( f(x) + r(x)\big) + \|c(x)\|_2.
$$
During the $k$th iteration, we want to choose $\tau_k$ such that $\tau_k \leq \tau_{k-1}$ and $s_k$ is a direction of sufficient descent for the merit function $\Phi_{\tau_k}(\cdot)$ at $x_k$. To define an appropriate value for $\tau_k$, let us define the model of the merit function given by
$$
q_k(s,\tau)
:= \tau\big(f_k + g_k^T s + \tfrac{1}{2\alpha_k} \|s\|_2^2 + r(x_k+s)\big) + \|c_k + J_k s\|_2,
$$
as well as the change in the model 
\begin{equation}\label{def:Deltaell}
\begin{aligned}
\Delta q_k(s,\tau)
&:= q_k(0) - q_k(s) \\
&\phantom{:}= -\tau\big( g_k^T s  + \tfrac{1}{2\alpha_k} \|s\|_2^2 +  r(x_k+s) - r_k\big) + \|c_k\|_2 - \|c_k+J_k s\|_2.
\end{aligned}
\end{equation}
Then, with parameters $\sigma_c\in(0,1)$ and $\sigma_u \in (0,\thalf]$, we set $\sigmabar_u := \sigma_u + \thalf \in(\thalf,1]$ and
\begin{align*} 
\tauktrial 
\gets
\begin{cases}
\infty & \text{if $g_k^T s_k + \tfrac{\sigmabar_u\|s_k\|_2^2}{\alpha_k} + r(x_k+s_k) - r_k  \leq 0$,} \\
\frac{(1-\sigma_c)(\|c_k\|_2 - \|c_k+J_k v_k\|_2)}{g_k^T s_k + \tfrac{\sigmabar_u\|s_k\|_2^2}{\alpha_k} + r(x_k+s_k) - r_k} & \text{otherwise,}
\end{cases}
\end{align*}
and then set, with $\epsilon_\tau\in(0,1)$, the $k$th merit parameter value as
\begin{equation} \label{update:tau}
\tau_k \gets
\begin{cases}
\tau_{k-1} & \text{if $\tau_{k-1} \leq \tauktrial$,} \\
\min\{(1-\epsilon_\tau)\tau_{k-1},\tauktrial\} & \text{otherwise.}
\end{cases}
\end{equation}
This update ensures that if the merit parameter is decreased during the $k$th iteration, it is decreased by at least a fraction of its previous value.  Moreover, the value for $\tauktrial$ ensures that $\Delta q_k(s_k,\tau_k)$ is an upper bound for quantities related to measures of criticality for problem~\eqref{prob:main} (see Lemma~\ref{lem:Deltaqk-positive}).  Moreover, Lemma~\ref{lem:Deltaqk-positive} shows that $-\Delta q_k(s_k,\tau_k)$ is an upper bound for the directional derivative of $\Phi_{\tau_k}(\cdot)$ at $x_k$ in the direction $s_k$ (this result holds regardless of the value of the merit parameter). 

The $k$th iteration is completed by checking whether the merit function achieves sufficient decrease in Line~\ref{line:check-suff-decrease}, and then defining the next iterate and proximal parameter. Specifically, if sufficient decrease is observed in the merit function, then the trial step $s_k$ is accepted  (i.e., $x_{k+1} \gets x_k + s_k$) and the proximal parameter value is unchanged (i.e., $\alpha_{k+1} \gets \alpha_k$); otherwise, the trial step is rejected (i.e., $x_{k+1} \gets x_k$) and the proximal parameter value is decreased (i.e., $\alpha_{k+1} \gets \xi \alpha_k$ for some $\xi\in(0,1)$).  This updating scheme motivates the definition of the following index set:
\begin{equation}\label{def:S}
\Scal := \{k: x_{k+1} = x_k + s_k \},
\end{equation}
which contains the indices of the successful iterations associated with  Algorithm~\ref{alg:main}.

\balgorithm[!th]
\caption{A proximal-gradient algorithm for problem~\eqref{prob:main}.}
\label{alg:main}
\balgorithmic[1]
   \State \textbf{input:} $x_0\in\R{n}$, $\alpha_0 \in \R{}_{> 0}$, and $\tau_{-1} \in\R{}_{>0}$.
   \State \textbf{constants:} $\kappav\in\R{}_{>0}$, $\{ 
   \sigma_c,  \epsilon_\tau,\xi,\eta,\}\subset (0,1)$, and $\sigma_u\in(0,1/2]$. 
   \For{$k = 0,1,2,\dots$}
        \If{$J_k^T c_k \neq 0$}
           \State Compute $v_k$ as an approximate solution to~\eqref{prob:normal.step} satisfying~\eqref{conds:vk}.\label{line:vk}
        \Else
           \State Set $v_k \gets 0$.
           \If{$c_k \neq 0$}
               \State \textbf{return} $x_k$ (infeasible stationary point)\label{line:isp}
           \EndIf
        \EndIf   
        \State Compute $u_k$ as the unique solution to~\eqref{prob:prox.sqp.relaxed}. \label{line:uk}
        \State Set $s_k \gets v_k + u_k$. \label{line:sk}
        \If{$s_k = 0$}
           \State \textbf{return} $x_k$ (first-order KKT point)\label{line:kkt}
        \EndIf
        \State Compute $\tau_k$ using~\eqref{update:tau}.
        \State Compute $\Delta q_k(s_k,\tau_k)$ using~\eqref{def:Deltaell}.
        \If{$\Phi_{\tau_k}(x_k+s_k) \leq \Phi_{\tau_k}(x_k)  - \eta\Delta q_k(s_k,\tau_k)$}\label{line:check-suff-decrease}
           \State Set $x_{k+1} \gets x_k + s_k$ and $\alpha_{k+1} \gets \alpha_k$.
        \Else
           \State Set $x_{k+1} \gets x_k$ and $\alpha_{k+1} \gets \xi\alpha_k$.\label{line:alpha-decreased}
        \EndIf  
    \EndFor
    \ealgorithmic
\ealgorithm

\section{Analysis}\label{sec:analysis}

In this section, we prove convergence results for Algorithm~\ref{alg:main}.  Our first result shows that the normal step $v_k$ is zero if and only if  $J_k^T c_k$ is zero.


\begin{lemma}\label{lem:v-is-zero}
For all $k\in\N{}$, it holds that $v_k = 0$ if and only if $J_k^Tc_k = 0$.
\end{lemma}
\begin{proof}
If $J_k^Tc_k = 0$, it follows from~\eqref{conds:vk-2} that $v_k =0$.
To prove the reverse implication, suppose that $v_k = 0$.  Then it follows from~\eqref{conds:vk-3} that $m_k(v_k) \leq m_k(v_k^c)$, which combined with~\eqref{eq:cauchy-decrease} and $v_k = 0$ 
shows that
$
0 = m_k(0) - m_k(v_k)
\geq m_k(0) - m_k(v_k^c)
\geq \thalf \|J_k^T c_k\|_2^2\min\left\{ \tfrac{1}{1+\|J_k^T J_k\|_2}, \kappav\alpha_k\right\}.
$
Since $\alpha_k > 0$ for all $k\in\N{}$ and $\kappav\in\R{}_{>0}$, it follows that $J_k^T c_k = 0$, completing the proof.
%
%
\end{proof}

Concerning the computation of the tangential step $u_k$, it follows from the optimality conditions for the convex optimization problem~\eqref{prob:prox.sqp.relaxed} that $u_k$ and the resulting $s_k = v_k + u_k$ satisfy, for some $g_{r,k} \in \partial r(x_k+s_k)$ and $y_k\in\R{m}$, the equalities
\begin{equation}~\label{kkt:u}
    g_k + \tfrac{1}{\alpha_k} u_k +  g_{r,k} - J_k^T y_k = 0
\ \ \text{and} \ \ 
J_k u_k = 0.
\end{equation}
Multiplying the first equality by $u_k^T$ and using the second equality, it follows that
\begin{equation}\label{eq:uT}
(g_k+g_{r,k})^Tu_k + \tfrac{1}{\alpha_k}\|u_k\|_2^2 = 0.  
\end{equation}
These equations related to the tangential step $u_k$ will be useful in the analysis. 

\subsection{Finite termination}

In this section we justify the finite termination conditions in Algorithm~\ref{alg:main} given in line~\ref{line:isp} and line~\ref{line:kkt}.  In particular, we show that if Algorithm~\ref{alg:main} terminates in line~\ref{line:isp} then $x_k$ is an infeasible stationary point, and if termination occurs in line~\ref{line:kkt} then $x_k$ is a first-order KKT point for problem~\eqref{prob:main}.

\begin{theorem}\label{thm:ft}
    The following finite termination results hold for Algorithm~\ref{alg:main}.
    \begin{itemize}
    \item[(i)] If termination occurs in line~\ref{line:isp} then $x_k$ is an infeasible stationary point, i.e., $x_k$ satisfies $c_k \neq 0$ and $J_k^T c_k = 0$.
    \item[(ii)] If termination occurs in line~\ref{line:kkt} then $x_k$ is a first-order KKT point for~\eqref{prob:main}.
    \end{itemize}
\end{theorem}
\begin{proof}
To prove part (i), suppose that termination occurs in line~\ref{line:isp} so that $c_k  \neq 0$ and $v_k = 0$.  It follows from $v_k = 0$ and Lemma~\ref{lem:v-is-zero} that $J_k^Tc_k = 0$, as claimed.

To prove part (ii), suppose that termination occurs in line~\ref{line:kkt} so that $s_k = 0$.  Since by construction $v_k^T u_k = 0$, it also follows that $v_k = u_k = 0$.  It follows from $v_k = 0$ and Lemma~\ref{lem:v-is-zero} that $J_k^T c_k = 0$.  Since termination must not have occurred in line~\ref{line:isp}, we also know that $c_k = 0$.  It follows from $v_k = u_k = 0$ and~\eqref{kkt:u} that there exists $g_{r,k} \in \partial r(x_k+s_k) \equiv \partial r(x_k)$ and $y_k\in\R{m}$ so that $g_k + g_{r,k} - J_k^T y_k = 0$.  Combining this equality with $c_k = 0$ shows that $x_k$ is a first-order KKT point for problem~\eqref{prob:main}.
\end{proof}

Theorem~\ref{thm:ft} shows that if Algorithm~\ref{alg:main}  finitely terminates, then the vector $x_k$ returned has favorable properties.  Admittedly, although finite termination in line~\ref{line:isp} is not ideal, the existence of infeasible stationary points is something that every algorithm must contend with unless an appropriate constraint qualification is assumed.

\subsection{Non-finite termination}

In this section, we study the convergence properties of Algorithm~\ref{alg:main} when finite termination does not occur.   Therefore, given how Algorithm~\ref{alg:main}{ is constructed, we know in this section that, for all $k\in\N{}$, it holds that
\begin{equation}\label{eq:s-nonzero}
(i) \ s_k \neq 0 \ \ \text{and} \ \ 
(ii) \ \text{$J_k^T c_k = 0$ if and only if $c_k = 0$.}
\end{equation}
%
%
%





Our first goal is to prove a bound on the directional derivative of $\Phi_\tau(\cdot)$ at $x_k$ in the direction $s_k$. Given the Lipschitz constants $\lipg$ and~$\lipJ$ in Assumption~\ref{ass:basic}, it follows for all $t \in \R{}_{>0}$ from~\cite[equation~(19)]{CurtRobiZhou23} that
\begin{equation}\label{lip-consts}
\begin{aligned}
  f(x_k + t s_k) &\leq f_k + t g_k^T s_k + \tfrac{\lipg}{2} t^2\|s_k\|_2^2
  \ \ \text{and} \\ 
  \|c(x_k + t s_k)\|_2 &\leq \|c_k + t J_k s_k\|_2 + \tfrac{\lipJ}{2} t^2\|s_k\|_2^2. 
\end{aligned}
\end{equation}
The next result gives an upper bound on  the quantity $D_{\Phi_\tau}(x_k,s_k)$, which we use to denote 
the directional derivative of 
$\Phi_\tau(\cdot)$ at $x_k$ in the direction $s_k$.
\begin{lemma}\label{lem:dd-bound}
The directional derivative of the merit function satisfies
$$
D_{\Phi_\tau}(x_k,s_k)
\leq \tau \big( g_k^T s_k + r(x_k+s_k) - r_k \big) + \|c_k+J_ks_k\|_2 - \|c_k\|_2.
$$
\end{lemma}
\begin{proof}
For all $t \in\R{}_{>0}$, it follows from~\eqref{lip-consts} and the triangle inequality that
\begin{align*}
\|c(x_k+ts_k)\|_2 - \|c_k\|_2
&\leq \|c_k+tJ_ks_k\|_2 - \|c_k\|_2 + \tfrac{\lipJ}{2} t^2\|s_k\|_2^2 \\
&\leq t\|c_k+J_ks_k\|_2 + (1-t)\|c_k\|_2
 - \|c_k\|_2 + \tfrac{\lipJ}{2} t^2\|s_k\|_2^2 \\
&= t\|c_k+J_ks_k\|_2 - t\|c_k\|_2 + \tfrac{\lipJ}{2} t^2\|s_k\|_2^2.
\end{align*}
On the other hand, it follows from~\cite[Theorem~2.25]{BagiKarmMake14} that $D_r(x_k,s_k) \leq r(x_k+s_k) - r(x_k)$.  The conclusion follows from this result, the previous displayed equation after dividing by $t$ and taking the limit $t \searrow 0$, and the fact that $f$ is differentiable. 
\end{proof}

Combining the previous lemma with how the merit parameter $\tau_k$ is defined, allows us to prove that the change in the model $q_k(s_k,\tau_k)$ is an upper bound for quantities used in our ultimate convergence result.

\begin{lemma}\label{lem:Deltaqk-positive}
The choice of $\tau_k$ in~\eqref{update:tau} ensures that the direction $s_k$ satisfies
\begin{align*}
\Delta q_k(s_k,\tau_k)
&\geq \tfrac{\sigma_u\tau_k}{\alpha_k}\|s_k\|_2^2 + \sigma_c\big( \|c_k\|_2 - \|c_k + J_k v_k\|_2\big)
> 0 
\ \ \text{and} \\
D_{\Phi_{\tau_k}}(x_k,s_k)
&\leq - \tfrac{\sigma_u\tau_k}{\alpha_k}\|s_k\|_2^2 - \sigma_c\big( \|c_k\|_2 - \|c_k + J_k v_k\|_2\big)
< 0.
\end{align*}
\end{lemma}
\begin{proof}
The first result follows from~\eqref{update:tau}, definition of $\tauktrial$, and $J_ks_k = J_k v_k$ (recall that $J_k u_k = 0$ because of the constraint in~\eqref{prob:prox.sqp.relaxed}). The second result follows from Lemma~\ref{lem:dd-bound}, $\tfrac{1}{\alpha_k}\|s_k\|_2^2 \geq 0$, and the first result of this lemma.
\end{proof}

We now give a sufficient condition for a successful iteration (see~\eqref{def:S}) to occur.

\begin{lemma}\label{cond:for-k-in-S}
If $(1-\eta)\Delta q_k(s_k,\tau_k) \geq \thalf(-\tfrac{\tau_k}{\alpha_k} +  \tau_k \lipg + \lipJ)\|s_k\|_2^2$, then $k\in\Scal$.
\end{lemma}
\begin{proof}
It follows from~\eqref{lip-consts}, \eqref{def:Deltaell}, and the assumed inequality in this lemma that
\begin{align*}
&\phi_{\tau_k}(x_k+s_k) - \phi_{\tau_k}(x_k) \\
&= \tau_k\big( f(x_k+s_k) + r(x_k+s_k)\big) + \|c(x_k+s_k)\|_2 - \tau_k\big( f_k + r_k\big) - \|c_k\|_2.\\
&\leq \tau_k g_k^T s_k + \tau_k\big(r(x_k+s_k) - r_k \big) + \|c_k+J_ks_k\|_2 - \|c_k\|_2 + \thalf(\tau_k \lipg+\lipJ)\|s_k\|_2^2 \\
&= -\Delta q_k(s_k,\tau_k) - \tfrac{\tau_k}{2\alpha_k}\|s_k\|_2^2 + \thalf(\tau_k \lipg+\lipJ)\|s_k\|_2^2 \\
&= -\Delta q_k(s_k,\tau_k) + \thalf \big( -\tfrac{\tau_k}{\alpha_k} + \tau_k \lipg + \lipJ \big)\|s_k\|_2^2 
\leq -\eta \Delta q_k(s_k,\tau_k). 
\end{align*}
Therefore, it follows from Line~\ref{line:check-suff-decrease} of Algorithm~\ref{alg:main} that $k\in\Scal$, as claimed.
\end{proof}

The following result gives a bound on the decrease in linearized feasibility achieved by $s_k$ that is similar to that achieved by the Cauchy point in~\eqref{eq:cauchy-decrease}.
\begin{lemma}~\label{lem:constraint-violation-change}
The step $s_k = v_k + u_k$ satisfies
$$
\|c_k\|_2 - \|c_k + J_k s_k\|_2
= \|c_k\|_2 - \|c_k + J_k v_k\|_2
\geq \tfrac{1}{2\kappa_c}\|J_k^T c_k\|_2^2 \min\left\{\tfrac{1}{1+\kappaJ^2}, \kappav\alpha_k\right\}.
$$
\end{lemma}
\begin{proof}
From~\eqref{eq:consequence-of-basic-assumption}, \eqref{eq:cauchy-decrease}, \eqref{conds:vk-3}, 
and the constraint in~\eqref{prob:prox.sqp.relaxed}, we have
\begin{equation*}
\begin{aligned}
\thalf \|J_k^T c_k\|_2^2 &\min\left\{ \tfrac{1}{1+\kappaJ^2}, \kappav\alpha_k\right\} \\
&\leq \thalf \|J_k^T c_k\|_2^2 \min\left\{ \tfrac{1}{1+\|J_k^T J_k\|_2}, \kappav\alpha_k\right\} \\
&\leq m_k(0) - m_k(v_k^c) 
= \thalf\big(\|c_k\|_2^2 - \|c_k+J_kv_k^c\|_2^2 \big) \\
&= \thalf\big(\|c_k\|_2 + \|c_k+J_kv_k^c\|_2\big)\big(\|c_k\|_2 - \|c_k+J_kv_k^c\| _2\big) \\
&\leq \|c_k\|_2\big(\|c_k\|_2 - \|c_k+J_kv_k^c\| _2\big) \\
&\leq \kappa_c\big( \|c_k\|_2 - \|c_k + J_k v_k\|_2 \big) 
= \kappa_c \big( \|c_k\|_2 - \|c_k + J_k s_k\|_2 \big),
\end{aligned}
\end{equation*}
from which the desired result follows.
\end{proof}

We now begin investigating quantities related to the merit parameter.  The following result bounds the denominator in the definition of $\tauktrial$.

\begin{lemma}\label{lem:den-bound}
For all $k\in\N{}$, it follows that
$$
g_k^T s_k + \tfrac{\sigmabar_u\|s_k\|_2^2}{\alpha_k} + r(x_k+s_k) - r_k
\leq 
(\kappag + \kappagr)\|v_k\|_2 + \tfrac{\sigmabar_u\|v_k\|_2^2}{\alpha_k}.
$$
\end{lemma}
\begin{proof}
    With $g_{r,k}$ defined as in~\eqref{kkt:u}, it follows from convexity of $r$ that $r_k \geq  r(x_k+s_k) + g_{r,k}^T(-s_k)$.  Combining this inequality with $s_k = v_k + u_k$, $v_k^T u_k = 0$, $\sigmabar_u\in(\thalf,1]$, \eqref{eq:uT}, the Cauchy-Schwartz inequality, and~\eqref{eq:consequence-of-basic-assumption} it follows that
\begin{align*}
&g_k^T s_k + \tfrac{\sigmabar_u\|s_k\|_2^2}{\alpha_k} + r(x_k+s_k) - r_k \\
&\leq (g_k + g_{r,k})^T s_k + \tfrac{\sigmabar_u\|s_k\|_2^2}{\alpha_k} \\
&= (g_k + g_{r,k})^T v_k + \tfrac{\sigmabar_u\|v_k\|_2^2}{\alpha_k} +  (g_k + g_{r,k})^T u_k + \tfrac{\sigmabar_u\|u_k\|_2^2}{\alpha_k} \\
&\leq (g_k + g_{r,k})^T v_k + \tfrac{\sigmabar_u\|v_k\|_2^2}{\alpha_k} +  (g_k + g_{r,k})^T u_k + \tfrac{\|u_k\|_2^2}{\alpha_k} \\
&\leq \|g_k + g_{r,k}\|_2 \|v_k\|_2 + \tfrac{\sigmabar_u\|v_k\|_2^2}{\alpha_k} \\
&\leq (\kappag + \kappagr)\|v_k\|_2 + \tfrac{\sigmabar_u\|v_k\|_2^2}{\alpha_k},
\end{align*}
which completes the proof.
\end{proof}

We next show that the merit sequence is positive and monotonically decreasing.

 \begin{lemma}\label{lem:tau-mono}
 For all $k \geq 1$, it holds that $0 < \tau_k \leq \tau_{k-1}$.
 \end{lemma}
 \begin{proof}
 It is clear from $\tau_0 > 0$ and the update~\eqref{update:tau} that $\{\tau_k\}$ is monotonically decreasing, and therefore all that remains is to prove that $\tau_k > 0$ for all $k\in\N{}$.   It follows from Lemma~\ref{lem:den-bound} and the definition of $\tauktrial$ that $\tauktrial = \infty$ if $v_k = 0$, and so for such $k$ we have $\tau_k \gets \tau_{k-1}$.  Therefore, for the remainder we only need to consider 
$k\in\N{}$ such that $v_k \neq 0$.  For such $k\in\N{}$, we know from Lemma~\ref{lem:v-is-zero} that $J_k^T c_k \neq 0$.  The result $\tau_k > 0$ follows from this observation, \eqref{update:tau}, $\alpha_k > 0$, and Lemma~\ref{lem:constraint-violation-change}.
 \end{proof}

The first part of the next lemma shows that the merit parameter never needs to be decreased for iterations $k\in\N{}$ such that $J_k^T c_k = 0$.  On the other hand, for all $k\in\N{}$ such that $J_k^T c_k \neq 0$, the second part of the result gives a lower bound on how small the previous merit parameter could have been.

\begin{lemma}\label{lem:lower-bound-tau-1}
    The following merit parameter update holds for each $k\in\N{}\setminus\{0\}$.
    \begin{itemize}
        \item[(i)] If $J_k^T c_k = 0$, then $\tauktrial = \infty$ and $\tau_k \gets \tau_{k-1}$.
        \item[(ii)] There exists a constant $\epsilon_{\tau} > 0$ such that, for all $k\in\N{}$ satisfying $J_k^T c_k \neq 0$ and $\tau_k < \tau_{k-1}$, it holds that
        $\tau_{k-1} \geq \epsilon_{\tau}\|J_k^T c_k\|_2$.
    \end{itemize}
\end{lemma}
\begin{proof}
For part (i), it follows from $J_k^T c_k = 0$ and  Lemma~\ref{lem:v-is-zero} that $v_k = 0$.  This fact, Lemma~\ref{lem:den-bound}, and the definition of $\tauktrial$ show that $\tauktrial = \infty$, so that $\tau_k \gets \tau_{k-1}$.

For part (ii), it follows from~\eqref{update:tau},  Lemma~\ref{lem:constraint-violation-change}, Lemma~\ref{lem:den-bound}, the trust-region constraint in problem~\eqref{prob:normal.step}, 
and~\eqref{eq:consequence-of-basic-assumption} that if $\tau_k < \tau_{k-1}$, then 
\begin{equation}\label{eq:tprev-bound-initial}
\begin{aligned}
\tau_{k-1}
&> \frac{(1-\sigma_c)(\|c_k\|_2 - \|c_k+J_k v_k\|_2)}{g_k^T s_k + \tfrac{\sigmabar_u\|s_k\|_2^2}{\alpha_k} + r(x_k+s_k) - r_k} \\
&\geq 
\frac{
\tfrac{(1-\sigma_c)}{2\kappa_c}\|J_k^T c_k\|_2^2 \min\{\tfrac{1}{1+\kappaJ^2}, \kappav\alpha_k\}
}
{
(\kappag + \kappagr)\|v_k\|_2 + \tfrac{\sigmabar_u\|v_k\|_2^2}{\alpha_k}
} \\
&\geq
\frac{
(1-\sigma_c)\|J_k^T c_k\|_2^2 \min\{\tfrac{1}{1+\kappaJ^2}, \kappav \alpha_k\}
}
{
2\kappac(\kappag + \kappagr)\kappav\alpha_k\|J_k^T c_k\|_2 + \tfrac{\sigmabar_u\kappav^2\alpha_k^2\|J_k^T c_k\|_2^2}{\alpha_k}
} \\
&= \frac{
(1-\sigma_c)\|J_k^T c_k\|_2 \min\{\tfrac{1}{1+\kappaJ^2}, \kappav \alpha_k\}
}
{
2\kappac(\kappag + \kappagr)\kappav\alpha_k + \sigmabar_u\kappav^2\alpha_k\|J_k^T c_k\|_2
} \\
&\geq 
 \frac{
(1-\sigma_c)\|J_k^T c_k\|_2 \min\{\tfrac{1}{1+\kappaJ^2}, \kappav \alpha_k\}
}
{
2\kappac(\kappag + \kappagr)\kappav\alpha_k + \sigmabar_u\kappav^2\alpha_k\kappaJ\kappac
}. 
\end{aligned}
\end{equation}
It follows from~\eqref{eq:tprev-bound-initial} and the fact that $\{\alpha_k\}$ is monotonically nonincreasing that
\begin{equation*}
\tau_{k-1}
\geq
\begin{cases}
 \frac{
(1-\sigma_c)\|J_k^T c_k\|_2
}
{
2\kappac(\kappag + \kappagr) + \sigmabar_u\kappav\kappaJ\kappac
}
&
\text{if $\kappav\alpha_k \leq 1/(1+\kappaJ^2)$,} \\
\frac{
(1-\sigma_c)\|J_k^T c_k\|_2
}
{
2\kappac(1+\kappaJ^2)(\kappag + \kappagr)\kappav\alpha_0 + \sigmabar_u\kappav^2\alpha_0\kappaJ\kappac
} 
&
\text{otherwise,}
\end{cases}
\end{equation*}
which completes the proof.
\end{proof}





We now prove our first key convergence result.  In particular, we prove that there must exist a subsequence of the set of successful iterations over which $\{J_k^T c_k\}$ converges to zero.  This conclusion is relevant to our setting because, under a suitable constraint qualification,  if $\xbar$ is a local minimizer of $\thalf\|c(x)\|_2^2$, then $J(\xbar)^T c(\xbar) = 0$.



\begin{theorem}\label{thm:JTc-to-zero}
Let Assumption~\ref{ass:basic} hold.  Then, there exists a subsequence of the iterations $\Kcal\subseteq\N{}$ such that $\lim_{k\in\Kcal} J_k^T c_k = 0$.
\end{theorem}
\begin{proof}
For a proof by contradiction, suppose that there exists a $\kbar_1\in\N{}$ and $\epsilon > 0$ such that $\|J_k^T c_k\|_2 \geq \epsilon$ for all $k\geq \kbar_1$.  Then, it follows from Lemma~\ref{lem:lower-bound-tau-1} and $\tau_0 \in\R{}_{>0}$ that there exits $\taubar_1 > 0$ such that, for all $k\in\N{}$, it holds that $\tau_k \geq \taubar_1$. Moreover, since $\{\tau_k\}$ is monotonically nonincreasing and when $\tau_k < \tau_{k-1}$ the reduction is by at least a constant factor (see~\eqref{update:tau}), we know that there exists $\kbar_2 \geq \kbar_1$ and $\taubar_2 \geq \taubar_1$ such that $\tau_k = \taubar_2$ for all $k \geq \kbar_2$.  Combining this with $\Delta q_k(s_k,\tau_k) > 0$ (see Lemma~\ref{lem:Deltaqk-positive})  
and Lemma~\ref{cond:for-k-in-S} it follows that for all $k\geq \kbar_2$ such that $\alpha_k \leq \taubar_2/(\taubar_2 \lipg + \lipJ)$ it must also hold that $k\in\Scal$.  Since $\alpha_{k+1} < \alpha_k$ only when $k\notin\Scal$, it follows that there must exist $\alphabar\in\R{}_{>0}$ and $\kbar_3 \geq \kbar_2$ such that $\alpha_k = \alphabar$ for all $k \geq \kbar_3$.  To summarize, we have proved that for all $k\geq \kbar_3$ it holds that $\alpha_k = \alphabar$, $\tau_k = \taubar_2$, and $k\in\Scal$.  It now follows from line~\ref{line:check-suff-decrease} of Algorithm~\ref{alg:main} that $\Phi_{\taubar_2}(x_{k+1}) \leq \Phi_{\taubar_2}(x_k)  - \eta\Delta q_k(s_k,\taubar_2)$ for all $k\geq \kbar_3$.  Summing over all $k \geq \kbar_3$ and using~\eqref{eq:consequence-of-basic-assumption}
and Lemma~\ref{lem:Deltaqk-positive} we have
\begin{align*}
\Phi_{\taubar_2}(x_{\kbar_3}) - \taubar_2 \finf
&\geq \sum_{k\geq \kbar_3} \big( \Phi_{\taubar_2}(x_k) - \Phi_{\taubar_2}(x_{k+1}) \big) \\
&\geq \eta \sum_{k\geq \kbar_3} \Delta q_k(s_k,\taubar_2) \\
&\geq \eta\sum_{k\geq \kbar_3} \tfrac{\sigma_u\tau_k}{\alpha_k}\|s_k\|_2^2 + \sigma_c\big( \|c_k\|_2 - \|c_k + J_k v_k\|_2\big) \\
&> \eta\sigma_c\sum_{k\geq \kbar_3} \big( \|c_k\|_2 - \|c_k + J_k v_k\|_2\big).
\end{align*}
Since the summation of nonnegative terms is finite, we know that
$$
\lim_{k\to\infty} \big( \|c_k\|_2 - \|c_k + J_k v_k\|_2\big) = 0.
$$
This fact, Lemma~\ref{lem:constraint-violation-change}, and $\alpha_k = \alphabar$ for all $k\geq \kbar_3$ 
imply that $\lim_{k\to\infty} J_k^T c_k = 0$, which contradicts our earlier assumption that $\|J_k^T c_k\|_2 \geq \epsilon$ for all $k\geq \kbar_1$.
\end{proof}


The remainder of the analysis considers two settings that are characterized by whether a certain constraint qualification holds or not.

\subsubsection{Strong LICQ}\label{sec:strong-licq}


In this section we make the following assumption, which is closely related to the linear independence constraint qualification (LICQ).


\begin{assumption}\label{ass:J-singular-value}
The smallest singular values of $\{J_k\}$ are uniformly bounded away from zero, i.e., there exists $\sigmamin > 0$ such that, for all $k\in\N{}$, $\sigmamin(J_k) \geq \sigmamin$.
\end{assumption}

We can now prove a nontrival bound on the improvement in linearized infeasibility achieved by the trial step $s_k$ relative to the actual infeasibility. This result is critical when we prove a uniform lower bound on the sequence of merit parameters.



\begin{lemma}\label{lem:wp}
    If $J_k^T c_k \neq 0$, then $s_k$ 
    satisfies
    $\|c_k+J_k s_k\|_2
    \leq \rho_k \|c_k\|_2$
    where
    \begin{equation*}
    \rho_k := 
    \sqrt{
    \max
    \Big\{
    1 - \kappav\alpha_k\sigmamin^2,
    1 - \sigmamin^2/\kappaJ^2
    \Big\}
    }
    \in [0,1).
    \end{equation*}
\end{lemma}
\begin{proof}
It follows from~\cite[Section~4.1]{nocedal2006numerical} that the Cauchy step $v_k^c$ in~\eqref{def:vkc} satisfies 
\begin{equation}\label{eq:vkc-in-proof}
 v_k^{c} = -\beta_k^c J_k^Tc_k \text{ with } \beta_k^c=\min\left\{\frac{\|J_k^Tc_k\|_2^2}{\|J_kJ_k^Tc_k\|_2^2},\kappav\alpha_k\right\}.
\end{equation}
We now consider two cases.



\noindent\textbf{Case 1:}  $\|J_k^Tc_k\|_2^2 \leq \kappav\alpha_k \| J_kJ_k^Tc_k\|_2^2$.  In this case, the minimum in~\eqref{eq:vkc-in-proof} is the first term, and $J_k J_k^T c_k \neq 0$ since $J_k^T c_k \neq 0$ .  These facts, the inequality that defines this case, the Cauchy-Schwartz inequality, definition of $m_k(0)$, and Assumption~\ref{ass:J-singular-value} give
\begin{align}\label{eq:bd1}
m_k(v_k^c) 
&= m_k(0) - \tfrac{1}{2}\tfrac{\|J_k^Tc_k\|_2^4}{\|J_kJ_k^Tc_k\|_2^2} \leq \tfrac{1}{2}\|c_k\|_2^2 - \tfrac{1}{2}\tfrac{\|J_k^Tc_k\|_2^2}{\kappaJ^2} \\
&\leq \thalf\|c_k\|_2^2 - \thalf\tfrac{\sigmamin^2(J_k)}{\kappaJ^2} \|c_k\|_2^2 
\leq \thalf\left(1-\tfrac{\sigmamin^2}{\kappaJ^2}\right)\|c_k\|_2^2.
\end{align}

\noindent\textbf{Case 2:}  $\|J_k^Tc_k\|_2^2 > \kappav\alpha_k \| J_kJ_k^Tc_k\|_2^2$.  In this case, the minimum in~\eqref{eq:vkc-in-proof} is the second term. This fact, the previous inequality, definition of $m_k(0)$, and Assumption~\ref{ass:J-singular-value} give
    \begin{align*}
        m_k(v_k^c)
        &= m_k(0) - \kappav\alpha_k \|J_k^Tc_k\|_2^2 + \tfrac{1}{2}\kappav^2\alpha_k^2\|J_kJ_k^Tc_k\|_2^2\nonumber\\
        &\leq  m_k(0) - \kappav\alpha_k\|J_k^Tc_k\|_2^2 
          + \tfrac{1}{2} \kappav\alpha_k \|J_k^Tc_k\|_2^2\nonumber\\
        &=  m_k(0) - \thalf\kappav \alpha_k\|J_k^Tc_k\|_2^2
        \leq \thalf \left(1- \kappav\alpha_k\sigmamin^2\right) \|c_k\|_2^2.
    \end{align*}

By combining the final result for the two cases, we find that
$m_k(v_k^c) \leq \thalf\rho_k^2\|c_k\|_2^2$. Multiplying both sides of this inequality by two, taking the square root, and using~\eqref{conds:vk-3} and the fact that $c_k+J_ks_k = c_k +J_k v_k$ since $J_k u_k = 0$, it follows that
$$
\|c_k+J_k s_k\|_2 
= \|c_k+J_k v_k\|_2
\leq \|c_k + J_k v_k^c\|_2
\leq \rho_k \|c_k\|_2,
$$
which completes the proof.
\end{proof}

We may now prove that $\{\tau_k\}$ is bounded away from zero.

\begin{lemma}\label{lem:tau-bounded-below}
For all $k\in\N{}$, it holds that $\tauktrial \geq \tautrialmin$ with 
\begin{align*}
&\tautrialmin := \\
&\min\left\{
\frac{ (1-\sigma_c) \kappav\sigmamin^2
}{2\kappav\kappaJ
\left(
\kappag + \kappagr + \sigmabar_u\kappac\kappav\kappaJ
\right)},
\frac{ (1-\sigma_c)(\sigmamin/\kappaJ)^2
}{2\kappav\kappaJ
\left(
\kappag + \kappagr + \sigmabar_u\kappac\kappav\kappaJ
\right)\alpha_0}
\right\},
\end{align*}
which when combined with~\eqref{lem:lower-bound-tau-1} gives $\tau_k \geq \taumin := \min\{\tau_0,(1-\epsilon_\tau)\tautrialmin\}$.
%
\end{lemma}
\begin{proof}
%
We first prove a lower bound on $\tauktrial$.  Since it follows that $\tauktrial = \infty$ for all $k\in\N{}$ satisfying $J_k^T c_k = 0$ (see Lemma~\ref{lem:lower-bound-tau-1}(i)), we may assume without loss of generality that each $k\in\N{}$ satisfies $J_k^T c_k \neq 0$. Next, we see from Lemma~\ref{lem:den-bound}, the trust-region constraint, and~\eqref{eq:consequence-of-basic-assumption} that
\begin{align*}
&g_k^T s_k + \tfrac{\sigmabar_u\|s_k\|_2^2}{\alpha_k} + r(x_k+s_k) - r_k \\
&\leq (\kappag + \kappagr)\|v_k\|_2 + \tfrac{\sigmabar_u\|v_k\|_2^2}{\alpha_k} \\
&= (\kappag + \kappagr)\kappav\alpha_k\|J_k^T c_k\|_2 + \sigmabar_u\kappav^2\alpha_k\|J_k^T c_k\|_2^2 \\
&\leq (\kappag + \kappagr)\kappav\alpha_k \kappaJ\|c_k\|_2 + \sigmabar_u\kappav^2\alpha_k\kappaJ^2\|c_k\|_2^2 \\
&\leq (\kappag + \kappagr)\kappav\alpha_k \kappaJ\|c_k\|_2 + \sigmabar_u\kappac\kappav^2\alpha_k\kappaJ^2\|c_k\|_2 \\
&= 
\kappav\kappaJ
\left(
\kappag + \kappagr + \sigmabar_u\kappac\kappav\kappaJ
\right)\alpha_k\|c_k\|_2
\ \ \text{for all $k\in\N{}$.}
\end{align*}
On the other hand, we may use 
Lemma~\ref{lem:wp} to obtain
\begin{align*}
\|c_k\|_2 - \|c_k+J_k v_k\|_2
&\geq \|c_k\|_2 - \rho_k\|c_k\|_2
= (1-\rho_k)\|c_k\|_2
\ \ \text{for all $k\in\N{}$.}
\end{align*}
Using the above two bounds and the definition of $\tauktrial$, it follows that
\begin{align*}
\tauktrial
&\geq \frac{ (1-\sigma_c)(1-\rho_k)\|c_k\|_2 }{\kappav\kappaJ
\left(
\kappag + \kappagr + \sigmabar_u\kappac\kappav\kappaJ
\right)\alpha_k\|c_k\|_2} \\
&= \frac{ (1-\sigma_c)(1-\rho_k) }{\kappav\kappaJ
\left(
\kappag + \kappagr + \sigmabar_u\kappac\kappav\kappaJ
\right)\alpha_k}
\ \ \text{for all $k\in\N{}$.}
\end{align*}
Next, notice that it follows from the definition of $\rho_k$ that
\begin{align*}
1-\rho_k
= \tfrac{1-\rho_k^2}{1+\rho_k}
&\geq \tfrac{1-\max\{1-\kappav\alpha_k\sigmamin^2,1-\sigmamin^2/\kappaJ^2\}}{2} \\
&= \tfrac{1 - \big( 1- \min\{\kappav\alpha_k\sigmamin^2, \, \sigmamin^2/\kappaJ^2\}\big) }{2} \\
&= \tfrac{\min\{\kappav\alpha_k\sigmamin^2, \, \sigmamin^2/\kappaJ^2\}}{2}
\ \ \text{for all $k\in\N{}$.}
\end{align*}
Combining this result with the previous displayed equation shows that
$$
\tauktrial
\geq
\frac{ (1-\sigma_c)\min\{\kappav\alpha_k\sigmamin^2, \, (\sigmamin/\kappaJ)^2\}
}{2\kappav\kappaJ
\left(
\kappag + \kappagr + \sigmabar_u\kappac\kappav\kappaJ
\right)\alpha_k}
\ \ \text{for all $k\in\N{}$.}
$$
It follows from this inequality and the fact that $\alpha_k \leq \alpha_0$ for all $k\in\N{}$ that
$$
\tauktrial
\geq
\begin{cases}
\frac{ (1-\sigma_c) \kappav\sigmamin^2
}{2\kappav\kappaJ
\left(
\kappag + \kappagr + \sigmabar_u\kappac\kappav\kappaJ
\right)} 
& \text{if $\kappav\alpha_k\sigmamin^2 \leq (\sigmamin/\kappaJ)^2$,} \\
\frac{ (1-\sigma_c)(\sigmamin/\kappaJ)^2
}{2\kappav\kappaJ
\left(
\kappag + \kappagr + \sigmabar_u\kappac\kappav\kappaJ
\right)\alpha_0} 
& \text{otherwise,}
\end{cases}
$$
for all $k\in\N{}$, which proves our first result. 

The second result, namely the positive lower bound on $\{\tau_k\}$, follows from the first result, $\tau_0 \in \R{}_{>0}$, 
and~\eqref{update:tau}, which completes the proof.
%
\end{proof}

The positive lower bound on $\{\tau_k\}$ lets us prove a positive lower bound on $\{\alpha_k\}$.



\begin{lemma}\label{lem:alpha-bounded-below}
If $\alpha_k \leq 
\taumin/(\taumin \lipg + \lipJ)$, then $k\in\Scal$.  Therefore,
\begin{equation}\label{lb-alpha-dependence-on-tau}
\alpha_k 
\geq \alphamin 
:= 
\min\{\alpha_0, \xi\taumin/(\taumin \lipg + \lipJ)  \} > 0 \ \ \text{for all $k\in\N{}$,}
\end{equation}
and a bound on the number of unsuccessful iterations is given by
\begin{equation}\label{bd-unsuccesful}
|\{k: x_k \notin \Scal\}| 
\leq \max\left(0, 
\left\lceil
\frac{\log\Big(\frac{\taumin}{\alpha_0(\taumin \lipg + \lipJ)}\Big)}{\log(\xi)}
\right\rceil
\right).
\end{equation}
\end{lemma}
\begin{proof}
Suppose that $k\in\N{}$ satisfies $\alpha_k \leq \taumin/(\taumin \lipg + \lipJ)$.
Then it follows from the definition of $\alphamin$, Lemma~\ref{lem:tau-bounded-below}, and the fact that $\tau/(\tau \lipg + \lipJ)$ is a monotonically increasing function on the nonnegative real line as a function of $\tau$ that
$$
\alpha_k 
\leq \taumin/ (\taumin \lipg + \lipJ)
\leq  \tau_k / (\tau_k \lipg + \lipJ), 
$$
which after rearrangement shows that $-\tau_k/\alpha_k + \tau_k \lipg + \lipJ \leq 0$.  It follows from this inequality, Lemma~\ref{lem:Deltaqk-positive}, and $\eta\in(0,1)$ that 
$$
(1-\eta)\Delta q_k(s_k,\tau_k) > 0 \geq \thalf(-\tfrac{\tau_k}{\alpha_k} +  \tau_k \lipg + \lipJ)\|s_k\|_2^2,
$$
which together with Lemma~\ref{cond:for-k-in-S} shows that $k\in\Scal$, as claimed.  
We know from the result we just proved and the update strategy for $\{\alpha_k\}$ that the bound in~\eqref{lb-alpha-dependence-on-tau} holds.  

Finally, the first result we proved in this lemma and the updating strategy for $\{\alpha_k\}$ shows that the maximum number of unsuccessful iterations is the smallest nonnegative integer $n_u$ such that $\xi^{n_u}\alpha_0 \leq \taumin/(\taumin \lipg + \lipJ)$, which gives the final result.
%
\end{proof}

Our worst-case complexity result uses the 
KKT-residual measure
\begin{equation}\label{def:chik}
\chi_k := \max\{\|g_k + g_{r,k} - J_k^T y_k\|_2,\|c_k\|_2\},
\end{equation}
where we remind the reader that $g_{r,k}$ is given in~\eqref{kkt:u}. In proving our complexity result, it will be convenient to define the shifted merit function
$$
\phibar_{\tau}(x) := \tau\big(f(x) - \finf + r(x) \big) + \|c(x)\|_2,
$$
where $\finf$ is defined in Assumption~\ref{eq:consequence-of-basic-assumption}. We stress that the (typically) unknown value $\finf$ is never used in the algorithm statement or its implementation, only in our analysis.  The following results pertain to the shifted merit function.
\begin{lemma}\label{lem:taubar-props}
The following properties hold for the shifted merit function $\phibar_\tau$:
\begin{itemize}
\item[(i)] For all $\{x,y\} \subset\R{n}$ and $\tau\in\R{}_{>0}$, it holds that $\phibar_\tau(x) - \phibar_\tau(y) = \phi_\tau(x) - \phi_\tau(y)$.
\item[(ii)] For all $x\in\R{n}$ and $0 < \tau_2 \leq \tau_1$, it holds that $\phibar_{\tau_2}(x) \leq \phibar_{\tau_1}(x)$. 
\item[(iii)] The sequence $\{\phibar_{\tau_k}(x_k)\}$ is monotonically decreasing.
\end{itemize}
\end{lemma}
\begin{proof}
For part (i), it follows from the definitions of $\phibar_\tau$ and $\phi_\tau$ that 
\begin{align*}
\phibar_\tau(x) - \phibar_\tau(y)
&= \tau\big(f(x) - \finf + r(x)\big) + \|c(x)\|_2 - \tau\big(f(y) - \finf + r(y)\big) - \|c(y)\|_2 \\
&= \tau\big(f(x) + r(x)\big) + \|c(x)\|_2 - \tau\big(f(y) + r(y)\big) - \|c(y)\|_2 \\
&= \phi_\tau(x) - \phi_\tau(y),
\end{align*}
which proves part (i). For (ii), the definition of $\finf$ and nonnegativity of $r$ imply  that
$$
\phibar_{\tau_2}(x)
= \tau_2\big(f(x) - \finf + r(x)\big) + \|c(x)\|_2
\leq \tau_1\big(f(x) - \finf + r(x)\big) + \|c(x)\|_2
= \phibar_{\tau_1}(x),
$$
which proves (ii). Finally, for each $k\in\N{}$, it follows from Lemma~\ref{lem:tau-mono}, parts (i) and (ii) of the current lemma, and how $x_{k+1}$ is updated in Algorithm~\ref{alg:main}  that
$$
\phibar_{\tau_k}(x_k) - \phibar_{\tau_{k+1}}(x_{k+1})
\geq \phibar_{\tau_k}(x_k) - \phibar_{\tau_k}(x_{k+1})
= \phi_{\tau_k}(x_k) - \phi_{\tau_k}(x_{k+1}) \geq 0,
$$
which completes the proof of this theorem.
\end{proof}

We may now state our worst-case complexity result for Algorithm~\ref{alg:main}.

\begin{theorem}\label{thm:complexity}
Suppose that  Assumption~\ref{ass:basic} and Assumption~\ref{ass:J-singular-value} hold, and let $\epsilon \in \R{}_{>0}$ be given.  If $\{k_1,k_2\}\subset\N{}$ are two iterations with $k_1 < k_2$ such that $k\in\Scal$ and 
$\chi_k > \epsilon$ for all iterations $k_1 \leq k < k_2$, then it follows that
\begin{equation}\label{diff-of-its}
k_2 - k_1 \leq \left\lfloor \frac{\tau_0\big( f(x_0) - \finf + r(x_0) \big) + \|c(x_0)\|_2 }{\kappaPhi\epsilon^2} \right\rfloor
\end{equation}
with $\kappaPhi := \eta \min\{\sigma_u\taumin\alphamin,\tfrac{\sigma_c\sigmamin^2}{2\kappac(1+\kappaJ^2)}, \tfrac{\sigma_c\sigmamin^2\kappav\alphamin}{2\kappac} 
\}$. 
Moreover, the maximum number of iterations before 
$\chi_k \leq \epsilon$ for some iteration $k\in\N{}$ is
\begin{equation}\label{eq:main-complexity-bound}
\left(
\max\left\{0,
\left\lceil
\frac{\log\Big(\frac{\taumin}{\alpha_0(\taumin \lipg + \lipJ)}\Big)}{\log(\xi)}
\right\rceil
\right\}
 + 1 \right)
\left\lfloor \frac{\tau_0\big( f(x_0) - \finf + r(x_0) \big) + \|c(x_0)\|_2 }{\kappaPhi\epsilon^2} \right\rfloor. 
\end{equation}
\end{theorem}
\begin{proof}
Let $\{k_1,k_2\}\subset\N{}$ be as described in the theorem statement.
Then, it follows from Lemma~\ref{lem:tau-mono}, Lemma~\ref{lem:taubar-props}(i--ii), Line~\ref{line:check-suff-decrease} of Algorithm~\ref{alg:main},  Lemma~\ref{lem:Deltaqk-positive}, Lemma~\ref{lem:tau-bounded-below}, and \eqref{lb-alpha-dependence-on-tau} that the following inequalities hold for all $k_1 \leq k < k_2$: 
\begin{align*}
    \phibar_{\tau_k}(x_k) - \phibar_{\tau_{k+1}}(x_{k+1})
    &\geq \phibar_{\tau_k}(x_k) - \phibar_{\tau_k}(x_{k+1}) \\
    &= \Phi_{\tau_k}(x_k) - \Phi_{\tau_k}(x_{k+1}) \\
    &\geq \eta \Delta q_k(s_k,\tau_{k}) \\
    & \geq \eta\tfrac{\sigma_u\tau_{k}}{\alpha_{k}}\|s_k\|_2^2 + \eta\sigma_c\big( \|c_k\|_2 - \|c_k + J_k v_k\|_2\big) \\
    &\geq \eta\sigma_u\taumin\alphamin \big(\tfrac{\|s_k\|_2}{\alpha_k}\big)^2 + \eta\sigma_c\big( \|c_k\|_2 - \|c_k + J_k v_k\|_2\big).
\end{align*}
Combining this inequality with $s_k = v_k + u_k$ and $v_k^Tu_k = 0$ for all $k\in\N{}$, Lemma~\ref{lem:constraint-violation-change}, Lemma~\ref{lem:alpha-bounded-below}, \eqref{kkt:u}, and Assumption~\ref{ass:J-singular-value} it follows, for all $k_1 \leq k < k_2$, that
\begin{align*}
    &\phibar_{\tau_k}(x_k) - \phibar_{\tau_{k+1}}(x_{k+1})  \\
	&\geq \eta\sigma_u\taumin\alphamin \big(\tfrac{\|u_k\|_2}{\alpha_k}\big)^2 + \eta\sigma_c \tfrac{1}{2\kappa_c} \|J_k^T c_k\|_2^2 \min\{ (1/(1+\kappaJ^2), \kappav\alphamin\} \\
	&\geq \eta\sigma_u\taumin\alphamin\|g_k + g_{r,k} - J_k^T y_k\|_2^2 + \eta\sigma_c \tfrac{\sigmamin^2}{2\kappa_c} \|c_k\|_2^2 \min\{ (1/(1+\kappaJ^2), \kappav\alphamin\} \\
	&\geq \kappaPhi \chi_k^2,
\end{align*}
where $\kappaPhi$ is defined in the theorem statement.
Using this inequality, Lemma~\ref{lem:taubar-props}(iii), and nonnegativity of $\phibar_{\tau}$ for all $\tau\in\R{}_{>0}$, 
we find that 
\begin{align*}
    \phibar_{\tau_0}(x_0)
    &\geq \phibar_{\tau_{k_1}}(x_{k_1}) 
    \geq \phibar_{\tau_{k_1}}(x_{k_1}) - \phibar_{\tau_{k_2}}(x_{k_2})  \\
    &\geq \sum_{k = k_1}^{k_2-1} \big(\phibar_{\tau_k}(x_k) - \phibar_{\tau_{k+1}}(x_{k+1}) \big) 
    \geq \sum_{k = k_1}^{k_2-1}\kappaPhi \chi_k^2, 
\end{align*}
which may then be combined with the fact that $\chi_k > \epsilon$ for all iterations $k_1 \leq k \leq k_2$ (see the assumptions of the current theorem) to conclude that
$$
\phibar_{\tau_0}(x_0)
\geq (k_2-k_1)\kappaPhi\epsilon^2,
$$
from which~\eqref{diff-of-its} follows.  

The final result in the theorem, namely the claimed upper bound on the maximum iterations before $\chi_k \leq \epsilon$, follows from what we just proved and the fact that maximum number of unsuccessful iterations is bounded as in~\eqref{bd-unsuccesful}.
\end{proof}

Before proving a result concerning convergence to a KKT point, we need to prove that the Lagrange multiplier estimates generated by subproblem~\eqref{prob:prox.sqp.relaxed} are bounded.
\begin{lemma}\label{lem:y-bounded}
The Lagrange multiplier estimate sequence $\{y_k\}$ is bounded.
\end{lemma}
\begin{proof}
Note from~\eqref{eq:uT} and the Cauchy-Schwarz and triangle inequalities that
$$
\tfrac{1}{\alpha_k}\|u_k\|_2^2
= -(g_k+g_{r,k})^Tu_k
\leq \|g_k + g_{r_k}\|_2\|u_k\|_2
\leq (\|g_k\|_2 + \|g_{r_k}\|_2) \|u_k\|_2,
$$
which when combined with~\eqref{eq:consequence-of-basic-assumption} shows that
\begin{equation}\label{bd-uoveralpha}
\tfrac{1}{\alpha_k}\|u_k\|_2 
\leq \kappag + \kappagr.
\end{equation}
Also observe that it follows from~\eqref{kkt:u} and Assumption~\ref{ass:J-singular-value} that
\begin{equation}\label{y-is-equal-to}
J_k^T y_k = 
g_k + \tfrac{1}{\alpha_k} u_k +  g_{r,k}
\iff
y_k = (J_kJ_k^T)^{-1}J_k\big( g_k + \tfrac{1}{\alpha_k} u_k +  g_{r,k}\big).
\end{equation}
Combining~\eqref{y-is-equal-to}, Assumption~\ref{ass:J-singular-value}, the triangle inequality, and \eqref{bd-uoveralpha}
it follows that
\begin{align*}
\|y_k\|_2
&\leq  \tfrac{1}{\sigmamin}\|g_k + \tfrac{1}{\alpha_k} u_k +  g_{r,k}\|_2 \\
&\leq \tfrac{1}{\sigmamin}
\Big(
\kappag + \kappagr + \tfrac{1}{\alpha_k} \|u_k\|_2
\Big) \\
&\leq \tfrac{2}{\sigmamin}
\Big(
\kappag + \kappagr 
\Big).
\end{align*}
Since this result holds for arbitrary $k\in\N{}$, we have proved the result.
\end{proof}

We can now prove that limit points of the primal sequence are KKT points.

\begin{theorem}\label{thm:limit-kkt}
Let Assumption~\ref{ass:basic} and Assumption~\ref{ass:J-singular-value} hold. Any limit point $\xstar$ of the sequence $\{x_k\}$ is a first-order KKT point for problem~\eqref{prob:main}, i.e., $c(\xstar) = 0$ and there exist vectors $\ystar\in\R{m}$ and $g_{r,*}\in \partial r(\xstar)$ such that $g(\xstar) + g_{r,*} - J(\xstar)^T\ystar = 0$.
\end{theorem}
\begin{proof}
Let $\xstar$ be a limit point of $\{x_k\}$, i.e., there exists $\Kcal_1$ so that $\{x_k\}_{k\in\Kcal_1}\to  \xstar$.  Theorem~\ref{thm:complexity} allows us to conclude that there exists a subsequence $\Kcal_2\subseteq\Kcal_1$ so that
\begin{equation}\label{chi-to-zero}
0 = \lim_{k\in\Kcal_2} \chi_k
=  \lim_{k\in\Kcal_2} \max\{\|g_k+g_{r,k}-J_k^T y_k\|_2,\|c(x_k)\|_2 \}.
\end{equation}
Lemma~\ref{lem:y-bounded} allows us to assert the existence of a vector $y_*\in\R{m}$ and subsequence $\Kcal_3\subseteq\Kcal_2$ such that $\{y_k\}_{k\in\Kcal_3} = y_*$. It follows from this limit, $\{x_k\}_{k\in\Kcal_3} \to x_*$,  continuity of $g$ and $J$, and~\eqref{chi-to-zero} that 
$$
\lim_{k\in\Kcal_3} g_{r,k}
= \lim_{k\in\Kcal_3} ( -g_k + J_k^T y_k ) 
= -g(x_*) + J(x_*)^T y_* =: g_{r,*}.
$$
Finally, combining this equality with $\{x_k\}_{k\in\Kcal_3} \to x_*$, continuity of $c$, and~\eqref{chi-to-zero} it follows that $g(x_*) + g_{r,*} - J(x_*)^T y_* = 0$ and $c(x_*) = 0$, which completes the proof.
\end{proof}


\subsubsection{Strong LICQ fails}

In this section we prove properties of the iterate sequence $\{x_k\}$ in Algorithm~\ref{alg:main} when the strong LICQ assumption used in the previous section (see Assumption~\ref{ass:J-singular-value}) does not hold. In such a setting, we should expect to prove weaker results since, for example, Lagrange multipliers may not even exist. 

Our main theorem of this section uses the quantity
\begin{equation}\label{def:chibark}
\chibar_k := \max\{\|g_k + g_{r,k} - J_k^T y_k\|_2, \|J^T_kc_k\|_2\},
\end{equation}
which is related to the quanity $\chi_k$ used in the previous section (see~\eqref{def:chik}). 

\begin{theorem}\label{thm:no-licq}
Let Assumption~\ref{ass:basic} hold. One of the following two cases occurs.
\begin{itemize}
\item[(i)] There exists $\taubarmin > 0$ such that $\tau_k \geq \taubarmin$ for all $k\in\N{}$.  In this case, it also follows that $\alpha_k \geq \alphabarmin 
:= 
\min\{\alpha_0, \xi\taubarmin/(\taubarmin \lipg + \lipJ)\}$ for all $k\in\N{}$ and,  
for a given $\epsilon > 0$, the maximum number of iterations before $\chibar_k \leq \epsilon$ is
\begin{equation*}
\!\!\!\!\!\!\!\!\!\!\!\!\left(
\max\left\{0,
\left\lceil
\frac{\log\Big(\frac{\taubarmin}{\alpha_0(\taubarmin \lipg + \lipJ)}\Big)}{\log(\xi)}
\right\rceil
\right\}
 + 1 \right)
\left\lfloor \frac{\tau_0\big( f(x_0) - \finf + r(x_0) \big) + \|c(x_0)\|_2 }{\kappabarPhi\epsilon^2} \right\rfloor
\end{equation*}
where $\kappabarPhi := \eta\min\{\sigma_u\taubarmin\alphabarmin, \tfrac{\sigma_c}{2\kappac(1+\kappaJ^2)}, \tfrac{\sigma_c\kappav\alphabarmin}{2\kappac} \}$.
\item[(ii)] The merit parameter values converge to zero, i.e., $\lim_{k\to\infty} \tau_k = 0$.  In this case, there exists a subsequence $\Kcal\subseteq\N{}$ such that $\lim_{k\in\Kcal} \|J_k^Tc_k\|_2 = 0$. 
\end{itemize}
\end{theorem}
\begin{proof}
Let us start by considering part (i), in which case we know that there exists $\taubarmin > 0$ such that $\tau_k \geq \taubarmin$ for all $k\in\N{}$.  Using this lower bound on $\{\tau_k\}$, the proof of Lemma~\ref{lem:alpha-bounded-below} still holds (with $\taumin$ replaced by $\taubarmin$), so that both~\eqref{lb-alpha-dependence-on-tau} and~\eqref{bd-unsuccesful} hold (with $\taumin$ replaced by $\taubarmin$), thus proving the first claim on $\alphabarmin$. Using~\eqref{lb-alpha-dependence-on-tau} and~\eqref{bd-unsuccesful} (with $\taumin$ replaced by $\taubarmin$), the proof of Theorem~\ref{thm:complexity} holds almost exactly as written. In particular, the proof holds as written until the middle of the second displayed equation, where we have (now with $\taumin$ and $\alphamin$ replace by $\taubarmin$ and $\alphabarmin$, respectively) that
\begin{align*}
&\phibar_{\tau_k}(x_k) - \phibar_{\tau_{k+1}}(x_{k+1})  \\
&\geq
\eta\sigma_u\taubarmin\alphabarmin\|g_k + g_{r,k} - J_k^T y_k\|_2^2
+ \eta\sigma_c \tfrac{1}{2\kappa_c} \|J_k^T c_k\|_2^2 \min\{ (1/(1+\kappaJ^2), \kappav\alphabarmin\}.
\end{align*}
If we now use the definitions of $\kappabarPhi$ and $\chibar_k$ we find that
$$
\phibar_{\tau_k}(x_k) - \phibar_{\tau_{k+1}}(x_{k+1})
\geq \kappabarPhi\chibar_k^2.
$$
The remainder of the proof of Theorem~\ref{thm:complexity} now follows exactly as written but with $\chibar_k$ and $\kappabarPhi$ in place of $\chi_k$ and $\kappaPhi$, respectively.  This completes the proof of part (i).


Part (ii) follows from Theorem~\ref{thm:JTc-to-zero}.
\end{proof}

A discussion on Theorem~\ref{thm:no-licq}(i) is of interest. In particular, the result in Theorem~\ref{thm:no-licq}(i) is of the same form as the result Theorem~\ref{thm:complexity}, with the only difference being the values of the constants $(\taumin,\alphamin,\kappaPhi)$ versus $(\taubarmin,\alphabarmin,\kappabarPhi)$.  A consequence of Assumption~\ref{ass:J-singular-value} used in Section~\ref{sec:strong-licq} is that we have an explicit definition for $\taumin$ (see Lemma~\ref{lem:tau-bounded-below}), which implies an explicit lower bound on $\alphamin$ and $\kappaPhi$ (see Lemma~\ref{lem:alpha-bounded-below} and Theorem~\ref{thm:complexity}).  On the other hand, no explicit lower bound on $\taubarmin$ is possible (in general) when Assumption~\ref{ass:J-singular-value} does not hold (in fact, it is even possible that $\{\tau_k\} \to 0$), and therefore the values for the constants $(\taubarmin,\alphabarmin,\kappabarPhi)$ in Theorem~\ref{thm:no-licq}(i) will depend on the particular value of $\taubarmin$ for that given problem.  In this respect, the complexity result of Theorem~\ref{thm:complexity} is stronger than Theorem~\ref{thm:no-licq}(i), which is not surprising since Theorem~\ref{thm:complexity} is proved under Assumption~\ref{ass:J-singular-value}.





\section{Numerical Results}\label{sec:numerical} 

In this section, we present results of  numerical experiments performed with our Python implementation of Algorithm~\ref{alg:main}.  The test problems are formulated with an $\ell_1$ regularizer, which is a common choice in many applications since it is known to induce sparse solutions.  The goal of our numerical tests is to validate the overall performance of our method using standard optimization metrics and to evaluate its ability to correctly identify the zero-nonzero structure of a solution.  For comparison purposes, we use the solver Bazinga~\cite{de2023constrained}, which is a safeguarded augmented Lagrangian method. The details concerning the test problems, our implementation, and the test results are given in the remainder of this section.

\subsection{Test problems}\label{sec:test-problems}

We considered a special instance of an $\ell_1$-regularized objective function with equality constraints that can be written in the form
\begin{equation}~\label{prob:test}
    \min_{x \in \R{n}, a \in \R{m}} \ \ f(x) + \lambda \|a\|_1 \ \ \text{s.t.} \ \ c(x) + a = 0  
\end{equation}
for some chosen regularization parameter $\lambda \in  \R{}_{>0}$. The functions $f$ and $c$ were chosen as a subset of the CUTEst~\cite{GouOT13} test problems, and we used PyCUTEst~\cite{fowkes2022pycutest} to evaluate these functions in our Python code. Our \emph{initial} test problems were chosen as the subset of CUTEst problems that satisfied the following properties: (i) the objective function was non-constant; (ii) the problem had at least one equality constraint, no inequality constraints, and no bound constraints on variables; and (iii) the number of equality constraints and variables satisfied $1 \leq m < n \leq 1000$. The restriction $m < n$ rules out problems that essentially reduce to finding a feasible point for the constraints, while the restriction $n < 1000$ is used to keep the computational cost to a manageable level.   As for the choice of $\lambda$, one can show that if $\xbar$ is a first-order KKT point with Lagrange multiplier vector $\ybar$ to the problem 
\begin{equation}\label{prob:equality}
    \min_{x \in \R{n}} \ \ f(x) \ \  \text{s.t.} \ \  c(x) = 0,
\end{equation} 
then $(\xbar,0)$ is a first-order KKT point to problem~\eqref{prob:test} with Lagrange multiplier $\ybar$ as long as $\lambda \geq \|\ybar\|_\infty$.  Therefore, in our tests, we set $\lambda = \|\ybar\|_\infty + 10$ where $\ybar$ is computed by solving problem~\eqref{prob:equality} using  IPOPT~\cite{wachter2006implementation}.  Since problems MSS1, MSS2, and CHAIN were not successfully solved by IPOPT, they were removed from the \emph{initial} test set, thus resulting in the \emph{final} set of 46 test problems 
found in Table~\ref{tab:appendix1}--Table~\ref{tab:appendix2}.  Although the problem formulation~\ref{prob:test} is somewhat contrived, this particular formulation allows us to better evaluate the structure identifying properties of the iterates produced by Algorithm~\ref{alg:main} and Bazinga.

\subsection{Implementation details} \label{sec:implementation}

The parameter and input values used are presented in Table~\ref{tab:alg-params} (no fine-tuning was performed). As for the starting point $(x_0, a_0)$ for problem~\eqref{prob:test}, the vector $x_0$ is set to the default value supplied by CUTEst and the vector $a_{0}$ is set as $-c(x_0)$ so that the initial point $(x_0,a_0)$ is feasible.

\begin{table}[ht]
\caption{Parameters and inputs used by Algorithm~\ref{alg:main}, with $x_0$ set to the value supplied by CUTEst.}
\label{tab:alg-params}
\centering
    \begin{tabular}{cccccccc}
    \hline
    $\alpha_0$ & $\tau_{-1}$ & $\kappa_v$ & $\sigma_c$ & $\epsilon_{\tau}$ & $\xi$ & $\eta$ & $\sigma_u$   \\
    \hline
    10 & 1 & 1000 & 0.1 & 0.1 & 0.5 & $10^{-4}$ & 0.1  \\
    \hline
    \end{tabular}
\end{table}

To approximately solve the trust-region subproblem~\eqref{prob:normal.step}, as needed in  Line~\ref{line:vk} of Algorithm~\ref{alg:main}, we used a Newton-like method. In particular, assuming for now that $J_k$ had full row-rank, we first computed the minimizer of $m_k(v)$ over all  $v\in \Range(J_k^T)$. Using the relationship $v = J_k^T w$, this problem may be written as
$$
\min_{w\in\R{m}} \ \thalf \|c_k\|_2^2 + w^T J_k J_k^T c_k + \thalf w^T J_k J_k^T J_k J_k^T w.
$$
It follows from the first-order optimality conditions and the full rank assumption on $J_k$ that the unique solution, call it $w_n$, satisfies
$$
\! J_kJ_k^TJ_kJ_k^T w_n = - J_kJ_k^T c_k
\iff
J_k J_k^T w_n = - c_k.
$$
After solving this linear system for $w_n$, we have that $v_n = J_k^T w_n$.  Next, we project this Newton step $v_n$ onto the trust-region constraint by defining 
\begin{align*}
    \vbar_n := \min\{ \|v_n\|_2, \kappav\alpha_k\|J_k^Tc_k\|_2 \} \frac{v_n}{\|v_n\|_2}.
\end{align*}
Also accounting for the possibility that $J_k$ may be rank deficient, we define $v_k$ as
$$
v_k \gets
\begin{cases}
v_k^c & \text{if $J_k$ does not have full rank or $m_k(v_k^c) < m_k(\vbar_n)$} \\
\vbar_n & \text{otherwise,}
\end{cases}
$$
which by construction ensures that $v_k$ satisfies conditions \eqref{conds:vk-1}-\eqref{conds:vk-3}, as needed.


Next, to solve subproblem~\eqref{prob:prox.sqp.relaxed} (as needed in Line~\ref{line:uk} of Algorithm~\ref{alg:main}) we exploit the structure of the $\ell_1$-norm.
By introducing variables  $(p,q) \in \R{n}_{\geq 0} \times \R{n}_{\geq 0}$ and using $e$ to denote the vector of all ones, we can consider the equivalent problem
\begin{equation}~\label{prob:tangential-qp}
    \begin{aligned}
    \min_{u,p,q} \ \ & g_k^T u + \tfrac{1}{2\alpha_k} \|u\|_2^2 + \lambda e^T(p + q) \\
    \text{s.t.} \ & J_k u = 0, \ \ 
     x_k + v_k + u = p - q, \ \ 
     p \geq 0, \ \ q \geq 0,
\end{aligned}
\end{equation}
which is a convex quadratic program (QP). To solve subproblem~\eqref{prob:tangential-qp} we use the primal active-set QP solver in the state-of-the-art software Gurobi version 11.0.0~\cite{gurobi}. Note that only a subset of the optimization variables receive $\ell_1$ regularization in the test problem formulation  (see~\eqref{prob:test}).  This setting is handled using the above scheme by introducing $p$ and $q$ variables only for those variables appearing in the $\ell_1$ norm.

Algorithm~\ref{alg:main} was terminated when one of the following conditions was satisfied.
\begin{itemize}
    \item \textbf{Approximate KKT point.}   Algorithm~\ref{alg:main} was terminated during the $k$th iteration with $x_k$ considered an approximate KKT point if $\|c_k\|_2 \leq 10^{-6}$ and $\|g_k+g_{r,k} - J_k^T y_k\|_2 \leq 10^{-6}$, as motivated by~\eqref{def:chik} and Theorem~\ref{thm:complexity}.
    \item \textbf{Approximate infeasible stationary point.} Algorithm~\ref{alg:main} was terminated during the $k$th iteration with $x_k$ considered an approximate infeasible stationary point if $\|c_k\|_2 \geq 10^{-2}$ and $\|J_k^T c_k\|_2 \leq 10^{-12}$.
     \item \textbf{Gurobi error.} Algorithm~\ref{alg:main} was terminated during the $k$th iteration if the Gurobi solver failed to solve subproblem~\eqref{prob:tangential-qp} using its default tolerances. 
    \item \textbf{Maximum iterations.} Algorithm~\ref{alg:main} was terminated if $1000$ iterations was completed without terminating for any of the previous reasons.
\end{itemize}

For comparison purposes, we solve the same test problems using the Bazinga method. Bazinga is a safeguarded augmented Lagrangian framework that uses an inner subproblem solver called $\mathrm{PANOC^{+}}$, which is a linesearch algorithm that combines a forward-backward iteration and a quasi-Newton step over the forward-backward envelop of the objective function; see the Bazinga paper~\cite{de2023constrained} for more details.\footnote{The code package of Bazinga is downloaded from~\url{https://github.com/aldma/Bazinga.jl}.}
The Bazinga algorithm was terminated when one of the following conditions was satisfied.
\begin{itemize}
    \item \textbf{Approximate KKT point.} Bazinga was terminated if certain primal feasibility and dual stationarity measures were less than $10^{-6}$. Further details on the termination conditions of Bazinga can be found in~\cite[Section 3.3]{de2023constrained}.
    \item \textbf{Not a number.}  Bazinga was terminated if a NaN occurred.
    \item \textbf{Maximum iterations.} Bazinga was terminated if $100$ iterations was completed without terminating for any of the previous reasons. Fewer maximum iterations was allowed for Bazinga compared to Algorithm~\ref{alg:main} because each iteration of Bazinga is significantly more expensive compared to Algorithm~\ref{alg:main}.  See the end of Section~\ref{sec:results} and Appendix~\ref{sec:appendix} for a discussion and table of results concerning computational times, respectively.
\end{itemize}

\subsection{Test results}\label{sec:results}

In this subsection, we present the results of using our Algorithm~\ref{alg:main} and Bazinga to solve problems of the form~\eqref{prob:test} with test functions chosen as described in Section~\ref{sec:test-problems}. To see detailed results for each test problem, see Table~\ref{tab:appendix1} and Table~\ref{tab:appendix2} in Appendix~\ref{sec:appendix}.  In the remainder of this section, we discuss the key results and observations summarized in Table~\ref{tab:comparison-table}.

We begin by describing the meanings of the columns of Table~\ref{tab:comparison-table}, and discuss their corresponding values to compare the performances of Algorithm~\ref{alg:main} and Bazinga.
\begin{itemize}
    \item \textbf{Method.}
    The name of the method.  
    \item \textbf{Feasible.}
    The number of test problems for which the corresponding method terminated at a point with constraint violation no larger than $10^{-6}$. For this metric we see that the two methods behaved similarly, with Bazinga achieving approximate feasibility on one more test problem.
    \item \textbf{Feasible, Better Objective.}
    To understand the meaning of this column, let $\fus$ denote the final objective value returned by Algorithm~\ref{alg:main} and $\fBazinga$ denote the final objective value returned by Bazinga.  We can then define the relative difference in the returned objective function values as
\begin{equation}~\label{eq:relative-diff-obj}
f_{\text{diff}} 
:=\frac{\fBazinga - \fus}{\max(1,|\min(\fBazinga,\fus)|)}. 
\end{equation} 
We indicate that Algorithm~\ref{alg:main} (resp., Bazinga) had a better relative objective value if $f_{\text{diff}} \geq 10^{-5}$ (resp., $f_{\text{diff}} \leq -10^{-5}$).  Using this terminology, column ``Feasible, Better Objective" gives the number of test problems for which both algorithms terminated at a point with constraint violation less than $10^{-6}$ \emph{and} the corresponding method had a better relative objective value. For this metric we see that Algorithm~\ref{alg:main} significantly outperforms Bazinga in terms of final objective function values when both algorithms return vectors that satisfy the constraint violation tolerance.
\item \textbf{Performs Better}.
The number of test problems for which the corresponding method either (i) met the constraint violation tolerance and the other method did not, or (ii) both methods reached the constraint violation tolerance and the corresponding method had a better relative objective value (see~\eqref{eq:relative-diff-obj}). Algorithm~\ref{alg:main} significantly outperforms Bazinga on this metric.
\item $a$ \textbf{is Zero}
The number of test problems for which the  corresponding method returned $a=0$. Algorithm~\ref{alg:main} significantly outperformed Bazinga in this metric, with Algorithm~\ref{alg:main} (resp., Bazinga) returning $a = 0$ on $36$ (resp., $13$) of the test problems.
We conjecture that Bazinga's poor performance on this metric is due to its inner subproblem solver, which  sacrifices solution sparsity for faster convergence of its iterates by combining proximal-gradient calculations with quasi-Newton ideas (see~\cite{de2023constrained}).  We investigated the test problems that Algorithm~\ref{alg:main} did not return $a=0$ and a  Gurobi error was not encountered, and found that  
by increasing the regularization parameter, Algorithm~\ref{alg:main} would return solutions satisfying $a = 0$.

\item $a$ \textbf{is Small.}
The number of test problems for which the corresponding method returned $\|a\|_\infty \leq 10^{-5}$, thus indicating that $a$ was small (possibly equal to zero).  When comparing this column with column ``$a$ is Zero", we see that 
the only difference is that Algorithm~\ref{alg:main} returned a small (nonzero) value for $a$ on one additional test problem; the results for Bazinga were unchanged. 
\item \textbf{KKT Found.}
The number of test problems for which the corresponding method terminated with an approximate KKT point, as discussed in Section~\ref{sec:implementation}.  Algorithm~\ref{alg:main} computed an approximate KKT point on $33$ of the $46$ test problems.  
Algorithm~\ref{alg:main} encountered Gurobi errors (see Section~\ref{sec:implementation}) on test problems BT4 and HS56 that were related to large constraint violation values, which were caused by too large of an initial value for the merit parameter. These failures can be avoided by decreasing the initial value for the merit parameter, but we did not do that for the numerical tests presented.  
\end{itemize}

\begin{table}[ht!]
\label{tab:comparison-table}
\centering
\caption{Algorithm~\ref{alg:main} versus Bazinga on various performance metrics related to solving problem~\eqref{prob:test} with test functions given in Table~\ref{tab:appendix1}--Table~\ref{tab:appendix2}; see Section~\ref{sec:results} for the meaning of the columns.}
\begin{tabular}{|c|c|c|c|c|c|c|}
\hline\rule{0pt}{0.8\normalbaselineskip}
Method & Feasible & Feasible, & Performs & $a$ is & 
$a$ is & KKT \\
       &          & Better Objective & Better & Zero  & Small & Found  \\
\hline
Algorithm~\ref{alg:main} & 40	& 23 & 23 & 36 & 37	& 33 \\
\hline
Bazinga & 41 & 2 & 7 & 13 & 13 & 21 \\
\hline
\end{tabular}
\end{table}

Overall, we are pleased with the results of  Table~\ref{tab:comparison-table}.  We believe that they indicate that there is significant merit to our proposed algorithm, especially in terms of computing structured approximate solutions.  It is worth noting that we have not discussed computational time since  comparing our Python implementation of Algorithm~\ref{alg:main} with the Julia implementation of Bazinga gives an advantage to Bazinga (purely because of the programming language used).  Even still, one can observe from Table~\ref{tab:appendix1} and Table~\ref{tab:appendix2} that Algorithm~\ref{alg:main} requires less (often significantly less) computing time compared to Bazinga on nearly every test problem.






\section{Conclusion}\label{sec:conclusion}

We have presented one of the first proximal-gradient type methods that can handle nonlinear equality constraints, and effectively return structured solutions where the structure is determined by the choice of regularization function. In the future, it would be interesting to address inequality constraints, establish convergence results under weaker assumptions, and accelerate convergence by  incorporating Nesterov acceleration or subspace acceleration.

\bibliographystyle{siamplain}


\clearpage
\appendix

\section{Detailed Results for the Test Problems}\label{sec:appendix}

In this appendix we provide the detailed output from our Algorithm~\ref{alg:main} and Bazinga for the test problems in  Table~\ref{tab:appendix1} and Table~\ref{tab:appendix2}.  See Section~\ref{sec:numerical} for details on the problem formulation, the test functions used, and the implementation details.

The columns of Table~\ref{tab:appendix1} and Table~\ref{tab:appendix2} have the following meanings.
\begin{itemize}
\item \textbf{Problem.} The name of the test problem.  Specifically, the value in this column gives the name of the CUTEst test problem used to obtain the objective function $f$ and constraint function $c$ in the test problem formulation~\eqref{prob:test}.
\item \textbf{Method.} The name of the method used.
\item \textbf{Obj.} The value of the objective function in  problem~\eqref{prob:test} at the final iterate returned by the solver. 
\item \textbf{RE.} The relative error between the objective function value returned by the algorithm and the optimal objective function value.  In particular, if we let $(f+r)$ denote the objective function value returned by a solver on a particular problem and let  $(f+r)_{\text{opt}}$ denote the optimal objective value for that same problem (as determined by the CUTEst documentation), then we define the relative error for that method on that problem as
\begin{equation*}
\textbf{RE} = \frac{|(f+r) - (f+r)_{\text{opt}}|}{\max(1,|(f+r)_{\text{opt}}|)}.
\end{equation*}
\item $\boldsymbol{\|c(x)+a\|_2.}$ The value of $\|c(x) + a\|_2$ at the point returned by the solver. 
\item $\boldsymbol{\|a\|_\infty.}$ The value of $\|a\|_\infty$ at the point returned by the solver.
\item \textbf{Status.} A three letter string that indicates the outcome when the given method was used to solve the given test problem. In particular, the value ``Opt" means that the method returned a final iterate that was an \textbf{approximate KKT point} as described in Section~\ref{sec:implementation}. The value ``Max" indicates that the method reached its maximum allowed number of iterations as described under \textbf{Maximum Iterations} in Section~\ref{sec:implementation}. The value ``Err" only occurred for Algorithm~\ref{alg:main} and indicates that a Gurobi error occurred as described under \textbf{Gurobi error} in Section~\ref{sec:implementation}. Finally, the value ``NaN" only occurred for Bazinga and indicates that the data type not-a-number occurred.
\item \textbf{Time.} The computational time measured in seconds.
\end{itemize}






\begin{table}[ht!]
    \centering
    \begin{filecontents*}{output1.csv}
Prob,Method,Obj,Error,Constr,y,Status,Time
BT11,Alg.~\ref{alg:main},8.25E-01,1.58E-07,1.01E-13,0.00E+00,Opt,1.71E+00
,Bazinga,2.22E+01,2.79E+01,8.23E-07,6.62E-01,Max,3.50E+02
BT12,Alg.~\ref{alg:main},6.19E+00,3.15E-10,4.26E-10,0.00E+00,Opt,2.99E-01
,Bazinga,3.75E+02,5.95E+01,1.02E-07,1.24E+01,Max,3.33E+02
BT1,Alg.~\ref{alg:main},-9.00E+01,8.90E+01,4.77E-13,1.00E+00,Opt,3.35E-01
,Bazinga,-5.54E+01,5.44E+01,3.68E-08,6.08E-01,Opt,9.37E+00
BT2,Alg.~\ref{alg:main},1.02E+05,3.14E+06,2.24E-05,1.02E+04,Max,5.17E+00
,Bazinga,3.26E-02,2.83E-07,8.22E-07,0.00E+00,Opt,1.24E+01
BT3,Alg.~\ref{alg:main},4.09E+00,3.10E-06,4.71E-15,0.00E+00,Opt,2.58E-01
,Bazinga,3.41E+01,7.33E+00,3.91E-09,5.92E-01,Max,3.39E+02
BT4,Alg.~\ref{alg:main},-2.67E+31,5.87E+29,2.63E+21,1.84E+10,Err,3.24E-02
,Bazinga,4.00E+01,1.88E+00,1.22E-07,2.69E+00,Max,3.26E+02
BT5,Alg.~\ref{alg:main},9.62E+02,7.32E-11,4.05E-10,0.00E+00,Opt,1.31E-01
,Bazinga,1.03E+03,7.60E-02,7.06E-08,3.11E+00,Max,3.26E+02
BT6,Alg.~\ref{alg:main},2.77E-01,4.89E-07,1.82E-12,0.00E+00,Opt,6.57E-01
,Bazinga,2.65E+01,9.45E+01,1.32E-07,1.19E+00,Max,3.42E+02
BT7,Alg.~\ref{alg:main},3.96E+01,8.71E-01,2.44E-13,4.54E-01,Opt,4.60E+00
,Bazinga,9.26E+02,2.02E+00,9.74E-07,2.50E-01,Opt,1.27E+02
BT8,Alg.~\ref{alg:main},1.00E+00,2.61E-06,8.67E-13,2.37E-07,Max,7.83E+00
,Bazinga,1.00E+00,9.86E-10,1.31E-09,0.00E+00,Opt,9.57E+00
BT9,Alg.~\ref{alg:main},-1.00E+00,1.05E-11,1.63E-11,0.00E+00,Opt,1.46E-01
,Bazinga,2.60E+01,2.70E+01,5.87E-07,1.24E+00,Max,3.23E+02
BYRDSPHR,Alg.~\ref{alg:main},-4.68E+00,7.61E-08,6.55E-09,0.00E+00,Opt,7.89E-02
,Bazinga,6.27E+00,2.34E+00,2.23E-08,4.92E-01,Opt,2.06E+01
DIXCHLNG,Alg.~\ref{alg:main},1.59E+02,1.59E+02,1.98E-07,9.70E-01,Max,1.55E+01
,Bazinga,NaN,NaN,NaN,NaN,NaN,1.35E+01
ELEC,Alg.~\ref{alg:main},1.46E+04,2.08E-01,1.12E-06,2.97E+00,Max,2.24E+04
,Bazinga,1.58E+04,1.42E-01,8.22E-08,6.45E-01,Max,3.34E+04
EXTROSNBNE,Alg.~\ref{alg:main},-2.00E+00,2.00E+00,1.86E-06,0.00E+00,Max,1.38E+05
,Bazinga,4.11E+03,4.11E+03,3.32E-07,4.12E-01,Max,3.65E+04
GENHS28,Alg.~\ref{alg:main},9.27E-01,9.27E-01,2.47E-15,0.00E+00,Opt,2.07E+00
,Bazinga,9.27E-01,9.27E-01,1.01E-10,0.00E+00,Opt,9.35E+00
HS100LNP,Alg.~\ref{alg:main},6.81E+02,1.09E-10,2.27E-13,0.00E+00,Opt,6.93E-01
,Bazinga,7.26E+02,1.16E+02,1.17E-07,1.46E+00,Max,3.36E+02
HS111LNP,Alg.~\ref{alg:main},-4.78E+01,1.10E-03,4.99E-10,0.00E+00,Max,1.15E+01
,Bazinga,-5.22E+01,9.52E-02,7.12E-08,2.60E+00,Max,4.58E+02
HS26,Alg.~\ref{alg:main},1.09E+00,1.09E+00,5.12E-08,0.00E+00,Max,5.27E+00
,Bazinga,1.05E-13,1.05E-13,2.93E-10,0.00E+00,Opt,9.14E+00
HS27,Alg.~\ref{alg:main},4.00E-02,5.63E-10,9.48E-15,0.00E+00,Opt,3.48E-01
,Bazinga,4.00E-02,1.46E-07,1.46E-07,0.00E+00,Opt,9.12E+00
HS28,Alg.~\ref{alg:main},5.05E-11,5.05E-11,8.88E-16,0.00E+00,Opt,2.26E-01
,Bazinga,7.95E-17,7.95E-17,8.01E-10,0.00E+00,Opt,9.05E+00
HS39,Alg.~\ref{alg:main},-1.00E+00,1.05E-11,1.63E-11,0.00E+00,Opt,1.37E-01
,Bazinga,2.60E+01,2.70E+01,5.87E-07,1.24E+00,Max,3.25E+02
\end{filecontents*}
\csvreader[tabular=lccccccc, before table=\rowcolors{2}{red!25}{yellow!50}, table head=\hline\rowcolor{red!50!black}\color{white}Problem & \color{white}Method & \color{white}Obj & \color{white}RE & \color{white}$\|c(x)+a\|_2$ & \color{white} $\|a\|_{\infty}$ & \color{white}Status & \color{white}Time \\\hline,
head to column names]{output1.csv}{}%
{\Prob & \Method & \Obj & \Error & \Constr & \y & \Status & \Time}%
    \caption{Results for test problems BT11-HS39.}
    \label{tab:appendix1}
\end{table}




\begin{table}[ht!]
    \centering
    \begin{filecontents*}{output2.csv}
Prob,Method,Obj,Error,Constr,y,Status,Time
HS40,Alg.~\ref{alg:main},-2.50E-01,1.39E-10,5.91E-11,0.00E+00,Opt,1.05E-01
,Bazinga,1.43E+01,5.81E+01,9.97E-10,4.78E-01,Opt,1.25E+02
HS42,Alg.~\ref{alg:main},1.39E+01,2.68E-10,3.06E-14,0.00E+00,Opt,2.68E-01
,Bazinga,2.76E+01,9.90E-01,5.91E-10,1.13E+00,Opt,9.38E+01
HS46,Alg.~\ref{alg:main},5.73E-10,5.73E-10,4.09E-12,0.00E+00,Opt,1.84E+00
,Bazinga,NaN,NaN,NaN,NaN,NaN,3.09E+02
HS47,Alg.~\ref{alg:main},1.33E-05,1.33E-05,4.70E-10,0.00E+00,Max,8.48E+00
,Bazinga,4.21E+01,4.21E+01,1.27E-10,1.36E+00,Opt,3.84E+01
HS48,Alg.~\ref{alg:main},2.69E-11,2.69E-11,4.44E-16,0.00E+00,Opt,1.84E-01
,Bazinga,2.00E-13,2.00E-13,1.59E-10,0.00E+00,Opt,9.06E+00
HS49,Alg.~\ref{alg:main},2.20E-04,2.20E-04,1.81E-13,0.00E+00,Max,7.22E+00
,Bazinga,2.90E+01,2.90E+01,2.02E-07,1.43E+00,Max,3.34E+02
HS50,Alg.~\ref{alg:main},2.86E-11,2.86E-11,1.72E-14,0.00E+00,Opt,6.58E+00
,Bazinga,4.53E+00,4.53E+00,1.66E-09,7.65E-02,Max,3.33E+02
HS51,Alg.~\ref{alg:main},1.48E-11,1.48E-11,6.28E-16,0.00E+00,Opt,1.71E-01
,Bazinga,1.61E-15,1.61E-15,3.17E-09,0.00E+00,Opt,9.05E+00
HS52,Alg.~\ref{alg:main},5.33E+00,4.31E-12,7.69E-16,0.00E+00,Opt,1.37E+00
,Bazinga,5.33E+00,6.50E-07,2.98E-07,0.00E+00,Opt,9.15E+00
HS56,Alg.~\ref{alg:main},-5.81E+83,1.68E+83,4.07E+28,4.94E+27,Err,6.57E-02
,Bazinga,NaN,NaN,NaN,NaN,NaN,7.44E+01
HS61,Alg.~\ref{alg:main},-1.44E+02,1.55E-11,2.03E-13,0.00E+00,Opt,1.75E-01
,Bazinga,-1.44E+02,4.29E-09,4.48E-07,0.00E+00,Opt,9.10E+00
HS6,Alg.~\ref{alg:main},1.16E-10,1.16E-10,2.38E-11,0.00E+00,Opt,4.11E-01
,Bazinga,4.46E-14,4.46E-14,2.49E-09,0.00E+00,Opt,1.20E+01
HS77,Alg.~\ref{alg:main},2.42E-01,6.83E-09,3.16E-12,0.00E+00,Opt,5.32E-01
,Bazinga,NaN,NaN,NaN,NaN,NaN,1.12E+01
HS78,Alg.~\ref{alg:main},1.10E+01,4.77E+00,4.73E-10,1.00E+00,Max,8.44E+00
,Bazinga,5.44E+01,1.96E+01,3.93E-09,1.74E+00,Max,3.44E+02
HS79,Alg.~\ref{alg:main},7.88E-02,7.11E-10,6.50E-14,0.00E+00,Opt,3.51E+00
,Bazinga,2.82E+01,3.57E+02,3.74E-08,9.33E-01,Max,3.48E+02
HS7,Alg.~\ref{alg:main},-1.73E+00,2.44E-10,8.18E-11,0.00E+00,Opt,1.81E+00
,Bazinga,-1.73E+00,2.12E-08,1.29E-07,0.00E+00,Opt,9.11E+00
HS9,Alg.~\ref{alg:main},-5.00E-01,1.32E-09,1.78E-15,0.00E+00,Opt,7.98E-02
,Bazinga,-5.00E-01,9.02E-10,1.37E-08,0.00E+00,Opt,9.10E+00
LCH,Alg.~\ref{alg:main},-1.23E+02,2.77E+01,3.42E-06,9.47E+00,Max,1.84E+03
,Bazinga,1.11E+00,1.26E+00,7.89E-09,3.50E-01,Max,1.07E+03
MARATOS,Alg.~\ref{alg:main},-1.00E+00,2.00E+00,4.71E-10,0.00E+00,Opt,6.19E-02
,Bazinga,5.48E+00,4.48E+00,7.69E-11,6.10E-01,Opt,9.56E+00
MWRIGHT,Alg.~\ref{alg:main},2.50E+01,2.40E-01,1.09E-13,0.00E+00,Opt,6.92E-01
,Bazinga,2.67E+02,7.11E+00,3.41E-07,3.02E+00,Max,3.46E+02
ORTHREGB,Alg.~\ref{alg:main},2.09E-12,2.09E-12,9.39E-13,0.00E+00,Opt,5.19E-01
,Bazinga,3.90E+02,3.90E+02,6.22E-08,6.49E+00,Max,5.21E+02
S316-322,Alg.~\ref{alg:main},3.34E+02,1.29E-06,3.31E-08,0.00E+00,Opt,6.14E-02
,Bazinga,9.12E+02,1.73E+00,3.26E-07,9.80E-01,Opt,1.59E+01
SPIN2OP,Alg.~\ref{alg:main},1.23E-09,1.23E-09,1.22E-09,0.00E+00,Opt,6.55E+02
,Bazinga,2.09E+03,2.09E+03,2.02E-08,2.09E+00,Max,4.45E+03
STREGNE,Alg.~\ref{alg:main},4.93E-12,4.93E-12,1.11E-15,0.00E+00,Opt,1.99E-01
,Bazinga,NaN,NaN,NaN,NaN,NaN,1.82E+01
\end{filecontents*}
\csvreader[tabular=lccccccc, before table=\rowcolors{2}{red!25}{yellow!50}, table head=\hline\rowcolor{red!50!black}\color{white}Problem & \color{white}Method & \color{white}Obj & \color{white}RE & \color{white}$\|c(x)+a\|_2$ & \color{white} $\|a\|_{\infty}$ & \color{white}Status & \color{white}Time (s)\\\hline,
head to column names]{output2.csv}{}%
{\Prob & \Method & \Obj & \Error & \Constr & \y & \Status & \Time}%
    \caption{Results for test problems HS40-STREGNE.}
    \label{tab:appendix2}
\end{table}

\end{document}